\documentclass[11pt,twoside]{article}

\usepackage{fullpage}
\usepackage[top=1in, bottom=1.25in, left=1in, right=1in]{geometry}

\usepackage{fullpage, amsthm, amsmath, amssymb, hyperref, dsfont, mathtools,bm}
\usepackage{framed}
\usepackage{caption,subcaption,graphicx,rotating,multirow}
\usepackage{url}
\usepackage{color,xcolor}
\usepackage{enumitem}
\usepackage{epstopdf}

\theoremstyle{plain}
\newtheorem{theorem}{Theorem}
\newtheorem{definition}[theorem]{Definition}

\newtheorem*{remark}{Remark}
\newtheorem{lemma}[theorem]{Lemma}

\newtheorem{corollary}[theorem]{Corollary}

\newtheorem{proposition}[theorem]{Proposition}
 \newenvironment{proofof}[1]{ \emph{Proof of #1.}}{\hfill \rule{2mm}{2mm} 
 }
  
\newlength{\widebarargwidth}
\newlength{\widebarargheight}
\newlength{\widebarargdepth}

\makeatletter
\long\def\@makecaption#1#2{
        \vskip 0.8ex
        \setbox\@tempboxa\hbox{\small {\bf #1:} #2}
        \parindent 1.5em  
        \dimen0=\hsize
        \advance\dimen0 by -3em
        \ifdim \wd\@tempboxa >\dimen0
                \hbox to \hsize{
                        \parindent 0em
                        \hfil 
                        \parbox{\dimen0}{\def\baselinestretch{0.96}\small
                                {\bf #1.} #2
                                } 
                        \hfil}
        \else \hbox to \hsize{\hfil \box\@tempboxa \hfil}
        \fi
        }
\makeatother

\newcommand{\half}{\frac{1}{2}}

\newcommand{\normal}{\ensuremath{\mathcal{N}}}

\newcommand{\Fspace}{\ensuremath{\mathcal{F}}}

\newcommand{\bigo}{\ensuremath{\mathcal{O}}}

\renewcommand{\P}{\operatorname{\mathbb{P}}}
\newcommand{\E}{\operatorname{\mathbb{E}}}

\newcommand{\Z}{\mathbb{Z}}

\newcommand{\R}{\mathbb{R}}

\newcommand{\lbr}{\langle}
\newcommand{\rbr}{\rangle}

\newcommand{\indi}{\mathds{1}}

\newcommand{\rmd}{\mathrm{d}}

\def\BC{\begin{center}}
\def\EC{\end{center}}
\def\BIT{\begin{itemize}}
\def\EIT{\end{itemize}}
\def\BET{\begin{enumerate}}
\def\EET{\end{enumerate}}
\def\BEQ{\begin{equation}}
\def\EEQ{\end{equation}}

\long\def\comment#1{}

\newtheorem{innercustomthm}{Theorem}
\newenvironment{customthm}[1]
  {\renewcommand\theinnercustomthm{#1}\innercustomthm}
  {\endinnercustomthm}

\usepackage[mathscr]{euscript}

\newcommand{\overlap}{\textup{\texttt{overlap}}}
\newcommand{\risk}{\mathcal{R}}
\newcommand{\bX}{{\bm X}}
\newcommand{\bW}{{\bm W}}
\newcommand{\hatbX}{\widehat{{\bm X}}}

\def\epsb{\bar{\eps}}
\def\BEC{{\rm BEC}}
\def\Unif{{\sf Unif}}
\def\onu{\overline{\nu}}
\def\hnu{\hat{\nu}}
\def\sdec{\mbox{\rm\tiny dec}}
\def\se{\mbox{\rm\tiny e}}

\def\sTV{\mbox{\rm\tiny TV}}
\def\hbtheta{\hat{\boldsymbol\theta}}
\def\btheta{{\boldsymbol\theta}}
\def\bY{{Y}}
\def\homega{\hat{\omega}}
\def\ssum{\mbox{\tiny\rm sum}}

\def\scm{\mbox{\tiny\rm cm}}
\def\sreg{\mbox{\tiny\rm reg}}

\def\htheta{\hat{\theta}}
\def\hbtheta{\hat{\boldsymbol\theta}}

\def\cF{{\mathcal F}}
\def\prob{{\mathbb P}}
\def\reals{{\mathbb R}}
\def\naturals{{\mathbb N}}
\def\tx{\tilde{x}}
\def\tG{\tilde{G}}
\def\tE{\tilde{E}}
\def\tX{\tilde{X}}
\def\cX{{\mathcal X}}
\def\cY{{\mathcal Y}}
\def\ocX{\overline{\mathcal X}}
\def\ocY{\overline{\mathcal Y}}
\def\cS{{\mathcal S}}
\def\cG{{\mathcal G}}
\def\hOmega{\widehat{\Omega}}

\def\eps{\varepsilon}
\def\bfone{{\mathbf 1}}

\def\integers{{\mathbb Z}}
\DeclareSymbolFont{rsfs}{U}{rsfs}{m}{n}
\DeclareSymbolFontAlphabet{\mathscrsfs}{rsfs}

\def\cuP{\mathscrsfs{P}}
\def\cuQ{\mathscrsfs{Q}}

\def\bx{{\boldsymbol x}}
\def\by{{\boldsymbol y}}
\def\bzero{{\boldsymbol 0}}

\usepackage{color}

\numberwithin{equation}{section}
\numberwithin{theorem}{section}

\begin{document}

\title{On the computational tractability of statistical estimation\\ on amenable graphs}

\author{Ahmed El Alaoui\textsuperscript{*} \;\; and\;\;\;  Andrea Montanari\thanks{Department of Electrical Engineering and
  Department of Statistics, Stanford University}}

\date{}
\maketitle

\begin{abstract}
We consider the problem of estimating a vector of discrete variables $\btheta = (\theta_1,\cdots,\theta_n)$, based on
noisy observations $Y_{uv}$ of the pairs $(\theta_u,\theta_v)$ on the edges of a graph $G=([n],E)$. This setting comprises a broad family of statistical
estimation problems, including group synchronization on graphs, community detection, and low-rank matrix estimation. 

A large body of theoretical work has established sharp thresholds for weak and exact recovery, and sharp characterizations of the 
optimal reconstruction accuracy in such models, focusing however on the special case of Erd\"os--R\'enyi-type random graphs. 
The single most important finding of this line of work is the ubiquity
of an information-computation gap. Namely, for many models of interest,
a large gap is found between the optimal accuracy achievable by any statistical method, and the optimal accuracy achieved by known polynomial-time algorithms.
Moreover, this gap is generally believed to be robust to small amounts of additional side information revealed about the $\theta_i$'s.

How does the structure of the graph $G$ affect this picture? Is the information-computation gap a general phenomenon 
or does it only apply to specific families of graphs? 

We prove that the picture  is dramatically different for graph sequences converging to amenable graphs (including, for instance, $d$-dimensional grids).
We consider a model in which an arbitrarily small fraction of the vertex labels is revealed, and show that a linear-time local algorithm can achieve reconstruction
accuracy that is arbitrarily close to the information-theoretic  optimum.
We contrast this to the case of random graphs. Indeed, focusing on group synchronization on random regular graphs, we prove that the information-computation gap still persists even when a small amount of side information is revealed.
\end{abstract}

\section{Introduction}

Classical statistics focuses on problems in which a small number of
parameters needs to be estimated from data. As a consequence, it
is mostly unconcerned with computational complexity considerations. Fundamental limits to statistical estimation are proven on the basis of 
information-theoretic considerations.  On the contrary, in modern high-dimensional applications, it is not uncommon to
come across statistical models that require estimating simultaneously  thousands or even  millions of parameters. In this
setting, a large gap is often observed between information-theoretic limits and what is  achieved by the best known  polynomial-time algorithms.
Indeed, it is expected that no polynomial-time algorithm can achieve optimal statistical performance in general.
In specific classes of models, a precise information-computation gap has been conjectured on the basis of current knowledge (see, e.g., \cite{MezardMontanari,decelle2011asymptotic,richard2014statistical,lelarge2017fundamental,barbier2019optimal,celentano2019fundamental} and references therein).

As explained below, most of our understanding of this information-computation gap was developed by analyzing probabilistic models
with a high degree of exchangeability.  This suggests a natural question: \emph{Is the same gap 
present in models with other type of structures?}

Statistical estimation on graphs provides a rich and interesting setting to study this question. 
Let $G_n = (V_n,E_n)$ be a graph on $n$ vertices, $V_n=[n]$. Edges are assumed to be directed in an arbitrary way, i.e., they are ordered pairs $(u,v)\in V_n\times V_n$.
We associate to the vertices $u \in V_n$ random variables $\btheta = (\theta_{u})_{u\in V_n}\sim_{\text{iid}} \Unif(\cX)$, uniformly distributed on a finite alphabet $\cX$. For each edge $(u,v)\in E_n$, we observe  $Y_{uv} \in \cY$, where $\cY$ is also a finite alphabet. The observations are conditionally independent with 
$Y_{uv}| \btheta \sim Q(\,\cdot\,|\theta_{u},\theta_{v})$, where $Q$ is a probability kernel from $\cX\times \cX$ to $\cY$.
Given the edge observations $\bY$ (and, possibly, additional side information, see  below), the purpose is to estimate the vertex assignment $\btheta$. 

This model is general enough to include a broad variety of examples studied in the literature, including group synchronization, community detection, 
low-rank matrix estimation, and so on. 
As an example consider the  \emph{$\Z_q$--synchronization} problem (further examples are presented in Section \ref{sec:Examples}).
The unknown variables $(\theta_u)_{u\in V_n}$ are i.i.d.\ uniform 
in $\cX =\Z_q = \{0,\cdots,q-1\}$, which we identify with the cyclic group $\Z\slash q\Z$ with additive structure. 
Observations are noisy measurements of the difference between $\theta_u$ and $\theta_v$ for each edge $(u,v) \in E_n$:     
\begin{align}\label{Z_q_model}
Y_{uv} = 
\begin{cases}
\theta_u - \theta_v ~(\textup{mod} ~ q) &\mbox{with probability } 1-p,\\
w_{uv} &\mbox{with probability } p,
\end{cases}
\end{align}
where $(w_{uv})_{(u,v)\in E_n}$ is a collection of independent random variables $w_{uv}\sim\Unif(\cX)$, independent of $(\theta_u)_{u\in V_n}$.

In addition to the observations $\bY$, we consider independent observations $(\xi^{(\eps)}_u)_{u \in V_n}$ on the vertices of $G_n$:
\begin{align}\label{side_info}
\xi^{(\eps)}_{u} = 
\begin{cases}
\theta_{u} &\mbox{with probability } \eps,\\
\star &\mbox{with probability } 1-\eps,
\end{cases}
\end{align}
where $\star$ is a symbol not belonging to $\cX$, so that with probability $\eps$ the value of $\theta_u$ is directly observed. 
We will write $\cX_{\star} := \cX\cup\{\star\}$.
Following the information theory literature, we refer to this noise model as the Binary Erasure Channel, and denote it by 
\BEC($\epsb$). (It is customary to parametrize the \BEC ~by its erasure probability $\epsb = 1-\eps$.)
The parameter $\eps$ will be considered very small (eventually going to zero as $n$ becomes large).
The purpose of this side information is to break the occasional group symmetry (sign symmetry or cyclic shifts in the case of $\Z_q$) 
that would otherwise be preserved by the observations $Y$.

We consider two metrics for the estimation accuracy.
 In our first  definition, the goal is to estimate the $n \times n$ rank-one matrix $\bX_f$ whose entries are
 \begin{equation}\label{matrix_f}
 (X_f)_{u,v} := f(\theta_u)f(\theta_v), ~~~ u,v \in V_n,
 \end{equation} 
 where $f : \cX \mapsto \R$ is a given real-valued function. For instance by setting $f(\theta) = \bfone_{\theta=x}$ and then considering all
values of $x\in\cX$, this allows to estimate whether $\theta_u=\theta_v$ for each pair of vertices $u,v\in V_n$.
An estimator is a map $\hatbX:\cY^{E_n}\times \cX_{\star}^{V_n}\to \reals^{n\times n}$, i.e., 
a function of the observations $\bY$ and the side information $\xi^{(\eps)}$.  We evaluate its risk under the square loss
 \begin{equation}\label{risk}
\risk_n(\hatbX;f) := \frac{1}{n^2} \E\Big[\big\|\bX_f - \hatbX\big\|_{F}^2\Big]\, .
 \end{equation} 
 We denote by $\risk_n^{\text{Bayes}}(f)$
 the minimal achievable error, i.e., the one achieved by the posterior expectation 
 \begin{equation}\label{posterior_estimator}
 \hatbX^{\text{Bayes}} := \Big(\E\Big[f(\theta_u)f(\theta_v)|Y_{G_n}^{(\eps)}\Big]\Big)_{u,v \in V_n}.
 \end{equation}
(We have made use of the following notation: for a graph $G=(V,E)$, we denote by $Y^{(\eps)}_{G}$ the union of the vertex and edge observations over $G$: 
$Y^{(\eps)}_{G} = \{Y_{uv} : (u,v)\in E, \xi^{(\eps)}_u : u\in V\}$.)
Our second metric for estimation accuracy is the `overlap', and will be introduced in Section~\ref{sec:Amenable}, see Eq.~\eqref{overlap}.

Statistical estimation on graphs has motivated substantial amount of work. 
In this context, the first example of a statistical model with a large information-computation gap is probably the planted clique
problem \cite{jerrum1992large,alon2002concentration}. This can be recast in the general framework described above, with $G_n$
the complete graph over $n$ vertices (see Section \ref{sec:Examples}). Despite more than a quarter century of research, and the study
of increasingly powerful classes of algorithms \cite{feige2000finding,deshpande2013finding,barak2016nearly}, no known 
polynomial-time algorithm comes close to saturate the information-theoretic limits for  this problem.

In recent years, a much more refined picture of the information-computation gap has emerged, mainly through the careful analysis of 
a variety of models on sparse random graphs  (as well as models on dense graphs in a different noise regime than the hidden clique model).
We refer to Section \ref{sec:related_work} for a brief summary of this vast literature. In most of these models an information-computation
gap is observed, and has been precisely delineated. This gap is generally conjectured to remain unchanged if a small amount of side 
information is revealed\footnote{The careful reader will notice that this statement does not apply to the planted clique problem. If the label
of $\eps n$ random vertices is revealed (i.e., whether or not they belong to the clique), then it is easy to find planted cliques of size $k\gg (1/\eps)\log n$,
i.e., far below the best known polynomial algorithms for $\eps=0$. This behavior is however related to the fact that, in the planted clique problem, 
the labels' prior distribution is strongly dependent on $n$, as
revealed from the fact that the clique's size is sublinear in $n$.}, 
as in Eq.~\eqref{side_info}. As mentioned above, most of the theoretical work has focused however on random graphs (Erd\"os--R\'enyi 
random graphs, random regular graphs and their relatives). 
This motivates the following key question:
\begin{center}
\emph{Does an information-computation gap exist for statistical estimation on other types of graphs?}
\end{center}
In this paper, we consider the case of graph sequences that converge locally to amenable graphs. 
Roughly, these are graphs for which the boundary of large sets of vertices is negligible compared to their volume.
We refer to Section \ref{sec:Background} for a reminder on the relevant definitions. Our results are already
interesting for the simplest example of such graphs, namely large boxes $[1,L]\times \cdots\times [1,L]$ in the $d$-dimensional grid
$\integers^d$ (with $L = n^{1/d}$).

Our main finding is that no information-computation gap exists for such graphs (as long as the gap is defined in terms of polynomial- 
versus non-polynomial time algorithms). A specific formalization of this finding is given below, and proved in Section~\ref{sec:Amenable}. 
\begin{customthm}{A}\label{thm:Risk}
Let $f:\cX\to \reals$ be a function with $\E[f(\theta)] = 0$ for $\theta\sim \Unif(\cX)$.
Let $G_n = (V_n,E_n)$ be a sequence of finite graphs (with $|V_n| = n$) converging locally--weakly to a random rooted graph $(G,o)$ which is infinite, locally-finite, almost surely anchored--amenable and tame. 
Then for each $l\in\naturals$ there exists an estimator $\hatbX^{(l)}:\cY^{E_n}\times \cX_{\star}^{V_n}\to \reals^{n\times n}$, with runtime $\bigo(n^2)$, such that the following 
holds.
For almost every $\eps>0$, we have 
\[\lim_{l \to \infty} \lim_{n \to \infty} \big\{\risk_n(\hatbX^{(l)};f) - \risk_n^{\textup{Bayes}}(f)\big\}=0.\]
\end{customthm}
The notions of local--weak convergence, anchored--amenability and tameness will be defined in Section~\ref{sec:Background}. 
More in detail, we present the following contributions:
\begin{description}
\item[No information-computation gap on amenable graphs.]  Theorem \ref{thm:Risk} provides a concrete formalization
of the general statement that statistically optimal estimation can be performed using polynomial time algorithms on 
(asymptotically) amenable graphs. In fact, we will prove that this follows from a more fundamental result, establishing 
that the vertex marginals of the posterior $\prob\big(\btheta\;  |Y^{(\eps)}_{G_n}\big)$ can be computed to arbitrary accuracy
in polynomial time, for almost all
values of $\eps$, on asymptotically amenable graphs, cf. Section  \ref{sec:Amenable}. 

Note that approximating the Bayes estimator $\hatbX^{\text{Bayes}}$, Eq.\ \eqref{posterior_estimator}, requires to approximate the joint distribution of pairs
of well separated vertices. However, we will use a decoupling argument to reduce ourselves to the case of vertex marginals.
\item[Local algorithms.] Our proof that  vertex marginals can be computed efficiently follows from an even stronger, and somewhat surprising fact (as above, holding for almost all $\eps$). The marginal at a vertex $v$ can be well approximated by computing the marginal with respect to the posterior
given observations in a large constant-size ball centered at $v$. In other words, the marginal can be approximated by a local algorithm.
The reason for this phenomenon can be explained in information theoretic terms. We will prove that the average conditional mutual information between a random vertex in a region $S\subseteq V$, and the boundary of $S$, $I\big(\theta_v;\theta_{\partial S}|Y_{S}^{(\eps)}\big)$ is upper bounded by $|\partial S|/|S|$. 
Hence, for amenable graphs, the effect of the boundary information is generally negligible.
\item[Robust information-computation gap on random regular graphs.]  We provide a counter-example, by showing that the conclusions
at the previous points do not hold for random regular graphs, converging locally to $k$-regular trees, which are non-amenable.
As  mentioned above, several cases of statistical estimation problems have been observed
to present an information-computation gap, when the underlying graph is random. While this gap is often expected to be robust to side information about the vertices,
we are not aware of any result that explicitly establishes robustness---in the setting of the present paper. We consider the $\integers_q$--synchronization problem on 
random $k$-regular graphs. We prove that, for a large range of the model parameters and all $\eps$ small enough: $(i)$ There exists a statistical estimator
that achieves non-trivial reconstruction accuracy uniformly as $\eps\to 0$; $(ii)$ Local algorithms can only achieve accuracy that vanishes as $\eps\to 0$.
\end{description}

\section{Related literature}
\label{sec:related_work}

As mentioned in the introduction, large information computation gaps were observed in a number of statistical estimation problems, when the underlying structure is 
a random graph, the complete graph, or close relatives. An incomplete list includes community detection in the stochastic block model 
\cite{decelle2011asymptotic,massoulie2014community,mossel2018proof,abbe2017community},
high-dimensional linear regression and generalized linear models \cite{barbier2019optimal,celentano2019fundamental},
 low-rank matrix estimation and sparse principal component analysis
\cite{johnstone2009consistency,amini2008high,berthet2013complexity,ma2015sum,lelarge2017fundamental}, tensor principal component analysis \cite{richard2014statistical,hopkins2015tensor,hopkins2017power}, tensor decomposition, and so on.

In many of these models, two types of results are established. On one hand an `information-theoretic' analysis allows to characterize the optimal
statistical accuracy that is achieved by an ideal estimator. On the other, specific classes of polynomial-time algorithms are analyzed.
Sometimes the resulting statistical estimation limits are stated in terms of specific goals such as `weak recovery' or `exact recovery': in the
present paper we consider the general goal of estimation with certain expected accuracy, or risk.

The most frequently analyzed classes of algorithms have been spectral methods, local algorithms, and convex relaxations in the sum-of-squares hierarchy. 
A remarkable dichotomy has emerged from these works. Roughly speaking, in all the examples  we know of, either highly sophisticated semidefinite programming hierarchies
fail, or simple combinations of spectral methods and local algorithms succeed. The behavior of the latter is in turn characterized by studying
the Bayes optimal local algorithm (belief propagation), in the presence of a small amount of side information.
Partial rationalizations of this surprising dichotomy were given in \cite{hopkins2016fast,hopkins2017power,fan2017well}. 
Motivated by this work, our analysis of $\integers_q$--synchronization on
random regular graphs (Section \ref{sec:RandRegular}) will focus on the same simple algorithm: belief propagation in the presence of side information.
As common in the literature, we will use the weak recovery threshold for this algorithm as a proxy for the fundamental algorithmic threshold.

Let us stress that our main focus is statistical estimation on amenable graphs.
Versions of this problem have been studied in a few  recent papers \cite{abbe2017group,sankararaman2018community,polyanskiy2018application,abbe2018information,abbe2018graph}. In particular, \cite{abbe2017group} proved the existence of a weak recovery threshold for 
$\integers_q$--synchronization\footnote{For $d=2$, \cite{abbe2017group} proves that a threshold exists in the case $q=2$, 
and indeed the same is expected to hold for $q\ge 3$ as well. For $d=1$ no non-trivial threshold exists in that weak recovery is always impossible.}
on grids in $d\ge 3$ dimensions.
However,  in contrast with random graphs, no explicit characterization exists (or is likely to exist) for the optimal statistical accuracy nor, in general, for the location of
weak recovery thresholds. This poses a clear challenge to us: we want to prove that the optimal statistical accuracy can be achieved by polynomial time algorithms, 
but \emph{we do not have an explicit characterization for the target accuracy}. Indeed, our proof will be purely conceptual.

Let us finally mention that it is well understood that certain algorithmic tasks are easy on graphs that can be embedded
well in $\reals^d$ (e.g., on grids). For instance, approximate optimization of a function that decomposes as a sum of edge terms over a grid is easy, by partitioning
the grid into large boxes. Unfortunately, these ideas do not have direct implications on the questions addressed in this paper. Even if we can find an 
approximate-maximum likelihood assignment of the unknown variables $\theta_i$, this is not guaranteed to have any good statistical properties, 
let alone achieve optimal estimation error. Inference and estimation do not reduce to optimization.

\section{Background}
\label{sec:Background}

\subsection{Further examples}
\label{sec:Examples}

It is interesting to check that the framework defined in the introduction is broad enough to encompass a variety of models of interest.

\vspace{0.25cm}

\noindent{\bf Spiked Wigner and Wishart models.} 
Low-rank plus noise models are ubiquitous in statistics and signal processing \cite{JohnstoneICM}, and can be recast in the language of the present paper.
As an example, consider the case of a signal vector $\btheta\in\reals^n$, with i.i.d.\ components, and assume we observe the rank-one-plus-noise
matrix  $\bY = \btheta \btheta^\top + \sigma_n\bW$. 
Here $\bW$ is a noise matrix, with --for instance-- $W_{uv}\sim \normal(0,1)$ and $\sigma_n$ controls the noise level. 

We take $G_n$ to be the complete graph, and $(\theta_u)_{u\in V_n}$ be i.i.d.\ random variables\footnote{Unlike for the model described in the introduction, the variables $\theta_u$'s
typically take any value in $\reals$, and their distribution is non-uniform. However, it is easy to reduce from one case to the other. For instance, 
we can let $Y_{uv} = \normal(h(\theta_u)h(\theta_v),\sigma^2_n)$. We can choose the nonlinear function $h:\reals\to\reals$ so that $h(\theta_v)\sim P_0$ when $\theta_v\sim\Unif([0,1])$.} 
from a distribution $P_{\theta}$ on $\R$. 
Observations on the edges are given by
\begin{align}
Y_{uv} \sim Q(\,\cdot\,|\theta_{u},\theta_v) = \normal(\theta_u\theta_v;\sigma_n^2 ) \, ,
\end{align} 
where $\normal(\mu,\sigma^2)$ denotes the Gaussian distribution.  

This example can be easily generalized. For instance, higher rank models can be produced by taking $\theta_u\in \reals^r$, $r\ge 1$ fixed. Rectangular (non-symmetric) 
random matrices of dimensions $n_1\times n_2$, can also be produced by setting $n=n_1+n_2$.
In this case  $\theta_v = (\zeta_v,b_v)$ where $\zeta_u\in\reals^r$ and $b_v\in \{1,2\}$ depending whether $v$ belongs to the first $n_1$ vertices (left factor) 
or the last $n_2$ ones (right factor).

\vspace{0.25cm}

\noindent{\bf Community detection.}  The stochastic block model is a popular model for community detection in networks.
The model is parametrized by a symmetric `connectivity' matrix  $(c_{rs})_{1\le r,s\le q}$, whereby $c_{r,s}\in [0,1]$ is the expected edge density between
vertices in communities $r$ and $s$. (For the sake of simplicity, we consider here the `balanced' case in which the $q$ communities have all equal expected size.)
Each vertex $v\in V_n$ is assigned a label $\theta_v\in [q]$ independently and uniformly at random. Conditional on $\btheta$, 
we generate a graph $\tG_n = (V_n,\tE_n)$ by connecting  vertices $u,v$ independently with probability $\prob((u,v)\in \tE_n|\btheta) = c_{\theta_u,\theta_v}$. 

We can encode this model in our general framework as follows.
 The graph $G_n$ is the complete graph, and observe $Y_{uv} \in \{0,1\}$ on every edge, where $Q(Y_{uv} = 1 | \theta_u=r,\theta_v=s) = c_{r,s}$.
The connection with the standard description is given by the correspondence $\{Y_{uv} = 1\}\;\; \Leftrightarrow\;\; \{(u,v)\in\tE_{uv}\}$.
The same encoding can be used for the planted clique problem.

\vspace{0.25cm}

Let us note that although the above models are special cases of our framework, we will focus in the rest of the paper onto graphs whose 
local--weak  limit (to be defined shortly) is locally finite.  This rules out graphs with diverging typical degree (in particular the complete graph).    

\subsection{Local--weak convergence and amenability}

For the reader's convenience, we collect here some relevant graph-theoretic definitions, referring to \cite{benjamini2001recurrence,aldous2007processes,lyons2017probability} for more details. In this paper, all graphs have a finite or countably infinite vertex set, are connected, and are \emph{locally finite}; i.e., all vertices have finite degree.   
A \emph{rooted} graph $(G,o)$ is a graph $G$ together with a choice of a vertex $o \in V(G)$, called the \emph{root} of $G$. We say that two rooted graphs $(G,o)$ and $(G',o')$ are isomorphic---and we write $(G,o) \equiv (G',o')$---if there exists an edge--preserving and root--preserving bijective map $ \phi : V(G) \mapsto V(G')$, i.e., $(u,v) \in E(G) \Leftrightarrow (\phi (u),\phi(v)) \in E(G')$, and $\phi(o) = o'$. 
For an integer $l\ge 0$, define $[G,o]_l$ to be the rooted subgraph spanned by a ball of radius $l$ around the root $o$ on $G$: this is the rooted graph $((V_l,E_l),o)$ where $V_l = B_G(o,l) := \{u \in V(G): d_{G}(o,u) \le l\}$, and $E_l = \{(u,v)\in E : u,v \in V_l\}$. Here, $d_G$ is the graph distance in $G$.
 
 \begin{definition} \label{def:local_convergence}
A sequence of rooted graphs $(G_n,o_n)_{n\ge 1}$ is said to \emph{converge locally} to a rooted graph $(G,o)$, and we write $(G_n,o_n) \xrightarrow[]{loc.} (G,o)$, if for every radius $l \ge 0$, there exists $n_0 \ge 0$ such that $[G_n,o_n]_{l} \equiv [G,o]_l$ for all $n \ge n_0$.
\end{definition}
This notion of convergence endows the set $\mathcal{G}_*$ (of $\equiv$--equivalence classes) of rooted graphs with a metrizable topology, called the topology of local, or Benjamini--Schramm, convergence~\cite{benjamini2001recurrence}. This gives $\mathcal{G}_*$ the structure of a complete separable metric space. Now we can define $\cuP(\mathcal{G}_*)$, the space of probability measures on $\mathcal{G}_*$ when endowed with its Borel $\sigma$--algebra. Then we endow $\cuP(\mathcal{G}_*)$ with the usual topology of weak convergence. 

From a \emph{finite} deterministic graph $G$, we can construct a random rooted graph $(G,o)$ by choosing the root $o$ uniformly at random from $V(G)$. We denote the law of this random rooted graph by $\rho_G \in \cuP(\mathcal{G}_*)$.     

\begin{definition}\label{def:local--weak}
A sequence of finite graphs $(G_n)_{n \ge 1}$ is said to converge locally--weakly to a random rooted graph $(G,o)$ if the sequence of probability measures $(\rho_{G_n})_{n \ge 1}$ converges weakly to a probability measure $\rho \in \cuP(\mathcal{G}_*)$, which is the law of $(G,o)$. 
\end{definition}
In other words, the definition requires that given a fixed finite connected rooted graph $(H,o')$ and a fixed radius $l$, the probability $\P\big([G_n,o_n]_l \equiv (H,o')\big)$ converges to $\P\big([G,o]_l \equiv (H,o')\big)$ as $n \to \infty$.    

Probability measures $\rho \in \cuP(\mathcal{G}_*)$ that are local--weak limits of sequences of \emph{finite} graphs as per Definition~\ref{def:local--weak} (such measures are called \emph{sofic} in the literature) inherit a important stationarity property which roughly expresses the intuition that the random graph $G$ should ``look the same" when viewed from any of its vertices. A formal definition takes the form of a \emph{mass--transport principle} termed \emph{unimodularity}~\cite{aldous2007processes}: Similarly to $\mathcal{G}_*$, we define $\mathcal{G}_{**}$ the space of $\equiv$--equivalence classes of \emph{doubly}--rooted graphs $(G,o,o')$ where the isomorphy relation $\equiv$ and local convergence as per Definition~\ref{def:local_convergence} are both extended in the natural way.   

\begin{definition} \label{def:unimodular}     
A measure $\rho \in \cuP(\mathcal{G}_*)$ is unimodular if for every Borel function $f:\mathcal{G}_{**} \to \R_+$,
\[\E_{\rho}\Big[\sum_{u \in V(G)} f(G,o,u)\Big] = \E_{\rho}\Big[\sum_{u \in V(G)} f(G,u,o)\Big],\]
when $(G,o)\sim \rho$.
\end{definition}
It is clear that if $G$ is finite then $\rho_{G}$ is unimodular, since the root is chosen uniformly at random. Furthermore, the property of unimodularity is closed in the topology of local--weak convergence~\cite{aldous2007processes}, hence \emph{all  local--weak limits of sequences of finite graphs are unimodular}.  

Next, we define the key concept of \emph{anchored--amenability}.
\begin{definition}\label{def:amenable}
An infinite rooted graph $(G,o)$ where $G=(V,E)$ is said to be anchored--amenable if its Cheeger constant anchored at $o$ is zero: 
\[\inf \Big\{|\partial S|/|S| : S \subset V \textup{ finite}, o \in S \Big\} = 0.\]
Here, $\partial S = \{u \in S: \exists v \notin S, (u,v)\in E\}$ is the vertex-boundary of the set $S \subseteq V$. 
\end{definition}
We will informally use the phrase `asymptotically amenable' to refer to graph sequences that converge locally--weakly to almost surely anchored--amenable graphs.

Observe that if $G$ is vertex--transitive, the above statement does not depend on the root $o$, and anchored--amenability reduces to the more classical notion of amenability of (non-rooted) graphs.    
For instance, the Euclidean lattice $\Z^d$ is amenable, the $k$-regular tree is not (both graphs being transitive).

Observe that if $(G,o)\sim \rho$ is almost surely anchored--amenable, there exists a sequence of finite sets $S_{k} \subset V$ such that $o\in S_k$ and which `witnesses' the amenability of $G$: $|\partial S_{k}|/|S_{k}| \longrightarrow 0$ as $k \to \infty$.
Moreover, this random sequence can be chosen in a \emph{measurable} way as a function of the rooted graph $(G,o)$. Indeed, we can for instance label the vertices of $G$ by $\naturals$, the root being labelled by $0$, and for every $k\ge 1$, choose the first finite set $S_k \subseteq V(G)$ (among countably many) in the lexicographic ordering such that $|\partial S_k|/|S_k| \le 2^{-k}$ and $o \in S_k$. For clarity we make this dependence explicit: $S_k = S_k(G,o)$. 
 We require a technical condition regarding such sets $S_k$.

\begin{definition}\label{def:tame}
We say that $\rho \in \cuP(\mathcal{G}_*)$ is \emph{tame} if it is supported on anchored--amenable rooted graphs, and there exists a sequence 
$\{S_k\}_{k\ge 1}$  of sets that witnesses anchored--amenability (i.e., such that $S_k(G,o)$ is a measurable function of $(G,o)$, and $|\partial S_k|/|S_k|\to 0$
almost surely) such that the following holds. For every $\eta>0$ there exists  $\delta>0$ such that
\begin{equation}\label{inverse_volume}
\limsup_{k \to\infty} \rho\Big((G,o)~:\sum_{u \in V(G)} \frac{ \bfone_{o\in S_k(G,u)}}{|S_k(G,u)|} \le \delta\Big) \le \eta.
\end{equation}
By extension, we say that the random rooted graph $(G,o)$ is tame if its law $\rho$ is tame.
\end{definition}

Intuitively, tameness is satisfied when the size of the neighborhoods $S_k(G,u)$ of each vertex $u$ around the root is
comparable with $S_k(G,o)$.  To discuss it further, it is useful to introduce the random variables
\begin{align}
\alpha_k(G,o) := \sum_{u\in V(G)} \frac{\bfone_{o \in S_k(G,u)}}{|S_k(G,u)|} \,.
\end{align}
The tameness condition requires a uniform upper bound on the lower tail of $(\alpha_k(G,o))_{k\ge 1}$.
An equivalent way to express this condition is to say that the sequence of random variables $(1/\alpha_k(G,o))_{k\ge 1}$ is \emph{tight} when $(G,o) \sim \rho$.

Note that  $\E_{\rho}[\alpha_k(G,o)]=1$ whenever $\rho$ is unimodular. Indeed, by a direct application of the mass-transport principle
(for the function $f(G,o,u) = \frac{\bfone_{o \in S_k(G,u)}}{|S_k(G,u)|}$)
\begin{align*}
\E_{\rho}[\alpha_k(G,o)] &= \E_{\rho}\left[\sum_{u\in V(G)} \bfone_{o \in S_k(G,u)}\frac{1}{|S_k(G,u)|} \right] \\
&= \E_{\rho}\left[\sum_{u\in V(G)} \bfone_{u \in S_k(G,o)}\frac{1}{|S_k(G,o)|} \right] =  \E_{\rho}\left[\frac{|S_k(G,o)|}{|S_k(G,o)|} \right] = 1\, .
\end{align*}
Moreover, tameness is satisfied if $\rho$ is supported on vertex-transitive graphs,
and in this case $\alpha_k(G,o)=1$ almost surely. Indeed, assume $\rho$ is supported on a single vertex-transitive graph.
Then $\rho$ is unimodular whence $\E_{\rho}[\alpha_k(G,o)]=1$, but $\alpha_k(G,o)$ is non-random and therefore $\alpha_k(G,o)=1$.
In the general case where $\rho$ is not an atom, since  $\alpha_k(G,o)=1$ almost surely conditional on $(G,o)$, we have $\alpha_k(G,o)=1$ almost surely unconditionally as well.

We next provide a few examples of graphs that are anchored-amenable and tame.

\vspace{0.25cm}

\noindent{\bf Example 1 (Percolation clusters).} Consider the $d$-dimensional
grid $\mathbb{L}^d$, i.e., $V(\mathbb{L}^d) = \integers^d$, and edges connect vertices at distance one $E(\mathbb{L}^d) = \{(\bx,\by)\in \integers^d:\; \|\bx-\by\|=1\}$.
Remove edges independently with probability $1-p$ and let $G=G_p(o)$ be the connected component of the origin $o=\bzero$.
We consider $p>p_c$, the percolation threshold on $\mathbb{L}^d$ so that $G$ is infinite with positive probability, and condition on the 
event that $G$ is indeed infinite.
In this case we can take $S_k(G,\bx)$ to be the subset of vertices contained in the $\ell_{\infty}$ ball of radius $\ell_k$ around $\bx$:
$S_k(G,\bx)= \{\by\in V(G): \|\by-\bx\|_{\infty}\le \ell_k\}$, for a deterministic sequence of radii $\ell_k\uparrow\infty$. A classical result of Newman and Schulman~\cite{newman1981number} implies $|S_k(G,o)|/\ell_k^d\to c_0$ almost surely for some non-random constant $c_0>0$. Further
$\partial S_k\subseteq\{ \bx\in \integers^d: \ell_k-1\le \|\bx\|_\infty\le \ell_k\}$ whence $|\partial S_k|\le c_1\ell_k^{d-1}$.
Hence, there exists a random $k_0<\infty$ such that almost surely $|\partial S_k|/|S_k|\le (2c_1/c_0)\, \ell_k^{-1}\to 0$, 
for all $k\ge k_0$.

Further, $|S_k(G,o)|\le c_2\ell_k^d$, and $o\in S_k(G,\bx)$ if and only if $\bx\in S_k(G,o)$. Therefore
\begin{align*}
\alpha_k(G,o)&\ge  \sum_{\bx\in V(G)} \bfone_{o\in S_k(G,\bx)}\frac{1}{c_2\ell_k^d} \\
&\ge  \frac{1}{c_2\ell_k^d} \sum_{\bx\in V(G)} \bfone_{\bx \in S_k(G,o)} = \frac{1}{c_2\ell_k^d} |S_k(G,o)|\, .
\end{align*}
Therefore $\underset{k\to\infty}{\liminf}\alpha_k(G,o) \ge c_0/c_2>0$ a.s., whence $\rho\big(\alpha_k(G,o)\le \delta\big) \to 0$
for all $\delta< c_0/c_2$.

\vspace{0.25cm}

\noindent{\bf Example 2 (Random geometric graph).} In this case the vertices are the points of a Poisson point process on $\reals^d$ with
constant intensity $\gamma$. Any two vertices $\bx,\by$ are connected by an edge if and only if $\|\bx-\by\|_2 \le r$ for a fixed radius $r>0$.
We choose the root $o\in V(G)$ as the closest vertex to the origin $\bzero$ and let $G$ be the connected component of  $o\in V$. This graph is infinite with positive probability provided $\gamma$ is larger than the percolation threshold $\gamma_c$ for this model~\cite{penrose2003random}.

The calculations for Bernoulli bond percolation on $\integers^d$ can be applied almost verbatim to the random geometric graph. In particular, letting $S_k(G,o) = \{\bx\in V(G):\, \|\bx\|_\infty \le \ell_k\}$ witnesses anchored--amenability and satisfies the tameness assumption.


\section{Results for asymptotically amenable graphs}
\label{sec:Amenable}

Recall that $Y^{(\eps)}_{G} $ refers to the union of the vertex- and edge-observations over $G$: $Y^{(\eps)}_{G} \:=\{Y_{uv} : (u,v)\in E, \xi^{(\eps)}_u : u\in V\}$.
A natural way to construct an estimator $\hbtheta$ is to first estimate the posterior marginals of $\btheta$ given $Y^{(\eps)}_{G_n}$ at every vertex:
\begin{equation}\label{marginals}
\mu_{G_n,u}(x) := \P\Big(\theta_{u} = x \big| Y^{(\eps)}_{G_n}\Big),~~ \mbox{for } u \in V_n \mbox{ and } x\in \cX.
\end{equation}   
Letting $(\widehat{\mu}_{u})_{u \in G_n}$ be such  estimates of the posterior marginals, we can construct $\hbtheta$, for instance, by independently sampling from the marginals: $\hat{\theta}_u \sim_{\text{ind}} \widehat{\mu}_{u}$, for all $u \in V_n$.

Of course, computing the exact posterior probabilities $\mu_{G_n,u}(x)$ is in general intractable. 
As a tractable  alternative, we can compute a \emph{local} version of the vertex marginals by using  only  observations in a ball of radius $l$ around each vertex.
For $u \in V_n$ and $x \in \cX$, let 
\begin{equation}\label{local_marginal}
\widehat{\mu}_{G_n,u,l}(x) := \P\Big(\theta_u = x \big| Y^{(\eps)}_{B_{G_n}(u,l)}\Big).
\end{equation}
(Recall that $B_{G_n}(u,l) = \{v \in V: d_{G_n}(u,v)\le l\}$ denotes the set of vertices within graph distance $l$ form $u$ in $G_n$.)  
The local marginals $\widehat{\mu}_{G_n,u,l}(x)$ can be computed with complexity at most $|\cX|^{|B_{G_n}(u,l)|}$ per vertex. The complexity
of estimating all the vertex marginals  is linear or nearly linear, under additional assumptions. In particular:
\begin{itemize}
\item If $G_n$ has degree bounded by $k_{\max}$ independently of $n$, then $|B_{G_n}(u,l)|\le k_{\max}^{l+1}$.  
\item If $G_n$ converges to a locally finite unimodular graphs, then 
\[\lim_{M\to\infty}\lim_{n \to \infty}\prob(|B_{G_n}(o_n,l)|\ge M) =0.\]
In other words, for each $\eps$, there exists $M(\eps)$ such that, for all $n$ large enough, all but a fraction
$\eps$ of the vertices $u$ have neighborhood of size bounded by $M(\eps)$. Hence $\mu_{G_n,u}(x)$ can be estimated for all but a fraction $\eps$ of the 
vertices in linear time.
\end{itemize}
Notice that we can safely neglect $o(n)$ atypical vertices for our purposes.
For instance, the matrix estimation risk \eqref{risk} is  bounded away in the present setting
(unless the channel $Q$ is noiseless), and therefore ignoring $o(n)$ vertices has a negligible impact on the asymptotic risk.

Do the local estimates $\widehat{\mu}_{G_n,u,l}$ provide good approximations of the actual marginals $\mu_{G_n,u}$?
Our  first result shows that this is the case for asymptotically amenable graphs, for almost all $\eps>0$, and on average over vertices in $G_n$.
\begin{customthm}{B}\label{main}
Let $G_n = (V_n,E_n)$ be a sequence of finite graphs (with $|V_n| = n$) that converges locally--weakly to random rooted graph $(G,o)\sim \rho$ which is almost surely anchored--amenable and tame.   
Then for almost every $\eps>0$,
\[\lim_{l \to \infty} \lim_{n \to \infty} \frac{1}{n} \sum_{u \in V_n} \E\big[d_{\sTV}(\widehat{\mu}_{G_n,u,l},\mu_{G_n,u})\big]= 0.\]
 \end{customthm}

The proof of this theorem follows from a technical result which we will present next. 

We define an observation model $(\btheta,\bY,\xi^{(\eps)})$ on the infinite random graph $G$ exactly as for the finite graphs $G_n$.  
We then let 
\[\mu_{G,o}(x) := \P\Big(\theta_o = x \big| (G,o), Y^{(\eps)}_{G}\Big),\] 
where we condition on the realization of the rooted graph and on $\sigma$-algebra generated by the sequence of 
random variables $\big(Y^{(\eps)}_{B_G(o,l)}\big)_{l\ge 0}$. 
Equivalently, we can also define $\mu_{G,o}(x)$ as the almost-sure limit of the sequence $\big(\P(\theta_o = x | (G,o), Y^{(\eps)}_{B_G(o,l)})\big)_{l\ge 0}$, where convergence is guaranteed by L\'evy's upward theorem. 
We have the following general relation between marginals on the finite graphs $G_n$, and marginals on the infinite rooted graph $(G,o)$.
 \begin{proposition}\label{L2_statement}
Under the conditions of Theorem~\ref{main}, we have for all $x \in \cX$ and almost every $\eps>0$,
 \begin{align}
 \lim_{l \to \infty} \lim_{n \to \infty} \frac{1}{n} \sum_{u \in V_n} \E\big[\widehat{\mu}_{G_n,u,l}^2(x)\big] &= \E\big[\mu_{G,o}^2(x)\big],\label{local}\\
\mbox{and} \quad 
\lim_{n \to \infty} \frac{1}{n} \sum_{u \in V_n} \E\big[\mu_{G_n,u}^2(x)\big] &= \E\big[\mu^2_{G,o}(x)\big].\label{global}
\end{align} 
(The expectation on the right-hand side is w.r.t.\ the randomness of $Y^{(\eps)}_G$ and $(G,o)$.)
\end{proposition}

The proof of Proposition~\ref{L2_statement} is presented in Section \ref{sec:Proof_L2Statement}. 
Theorem~\ref{main} is a consequence of Proposition~\ref{L2_statement} as shown below.
\begin{proof}[Proof of Theorem \ref{main}]
We claim that $\E\big[\widehat{\mu}_{G_n,u,l}(x)\mu_{G_n,u}(x)\big] = \E\big[\widehat{\mu}_{G_n,u,l}(x)^2\big]$. Indeed, by conditioning on $Y^{(\eps)}_{B_{G_n}(u,l)}$ we obtain 
\begin{align*} 
\E\big[\widehat{\mu}_{G_n,u,l}(x)\mu_{G_n,u}(x)\big] &= \E\big[ \P\big(\theta_u = x | Y^{(\eps)}_{B_{G_n}(u,l)}\big) \E\big[\P\big(\theta_u = x | Y^{(\eps)}_{G_n}\big) | Y^{(\eps)}_{B_{G_n}(u,l)}\big] \big]\\
&=\E\big[\P\big(\theta_u = x | Y^{(\eps)}_{B_{G_n}(u,l)}\big)^2\big] = \E\big[\widehat{\mu}_{G_n,u,l}(x)^2\big].
\end{align*}
Now we use the fact that for two measures $\mu$ and $\nu$ on $\cX$, $d_{\sTV}(\mu,\nu) = \half \|\mu - \nu\|_{\ell_1} \le \half \sqrt{|\cX|} \cdot \|\mu - \nu\|_{\ell_2}$:
\begin{align*}
\E\big[d_{\sTV}(\widehat{\mu}_{G_n,u,l},\mu_{G_n,u})\big]^2&\le \frac{1}{4} |\cX|\E\left[\|\widehat{\mu}_{G_n,u,l}-\mu_{G_n,u}\|_{\ell_2}^2\right]\\
&=\frac{1}{4} |\cX| \sum_{x\in\cX} \left(\E\big[\E\big[\mu_{G_n,u}(x)^2\big] - \widehat{\mu}_{G_n,u,l}(x)^2\big]\right)\, .
\end{align*}
(Here and below $\|\mu-\nu\|_{\ell_p}$ denotes the $\ell_p$ norm of the vector $(\mu(x)-\nu(x))_{x\in\cX}$.)
The claim follows by averaging over $u\in V_n$, and applying Proposition~\ref{L2_statement}.
\end{proof}

Note that Theorem~\ref{main} is not sufficient to establish Theorem \ref{thm:Risk} about the optimality of polynomial-time 
algorithms to estimate the pairwise correlations $(X_f)_{u,v} = f(\theta_u)f(\theta_v)$. Indeed, the latter requires to approximate 
the joint distribution of $\theta_u$, $\theta_v$ for $u, v\in V_n$ two arbitrary vertices.
In order to achieve this goal, we define a \emph{decoupled} estimator:
\begin{align}\label{estimated_matrix}
 \widehat{X}^{(\sdec)}_{uv} &:= \E\Big[f(\hat{\theta}_u)\big| Y^{(\eps)}_{G_n}\Big]\cdot\E\Big[f(\hat{\theta}_v)\big|Y^{(\eps)}_{G_n}\Big]\\
 &= \Big(\sum_{x\in \cX} \mu_{G_n,u}(x)f(x)\Big) \cdot \Big(\sum_{x\in \cX} \mu_{G_n,v}(x)f(x)\Big),~~~ u, v \in V_n.\nonumber
 \end{align}
Note that $\hatbX^{(\sdec)}$ may \emph{a priori} have suboptimal accuracy. This is however not the case for almost all $\eps$.
\begin{proposition}\label{decoupling_bayes}
Let $\hatbX^{(\sdec)}\in\R^{n \times n}$ be defined as per Eq.~\eqref{estimated_matrix}.
Then for almost every $\eps>0$, 
\[
\lim_{n \to \infty} \big\{\risk_n(\hatbX^{(\sdec)};f) - \risk_n^{\textup{Bayes}}(f)\big\} =0\, .
\]
\end{proposition}
The proof of the above proposition can be found in Appendix~\ref{sec:omitted}.

Given Theorem \ref{main} and Proposition \ref{decoupling_bayes}, it is natural to consider the following low complexity version of 
$\hatbX^{(\sdec)}$: 
\begin{align}
\widehat{X}^{(l)}_{uv} := \Big(\sum_{x\in \cX} \widehat{\mu}_{G_n,u,l}(x)f(x)\Big) \cdot \Big(\sum_{x\in \cX} \widehat{\mu}_{G_n,v,l}(x)f(x)\Big)\, .
\end{align}
Since we can compute $\widehat{\mu}_{G_n,u,l}(x)$ for all but $o(n)$ vertices in time $\bigo(1)$, the overall complexity of $\hatbX^{(l)}$ is $\bigo(n^2)$. 
(Setting $\widehat{X}^{(l)}_{uv}=0$ for a sublinear fraction of vertices produces a negligible error.)
We can now prove Theorem \ref{thm:Risk}. 

\vspace{.25cm}
\begin{proofof}{Theorem~\ref{thm:Risk}}
Since  Proposition~\ref{decoupling_bayes} yields $\risk_n(\widehat{\bX}^{(\sdec)};f) - \risk_n^{\textup{Bayes}}(f) \rightarrow 0$ for almost all $\eps>0$, we only need to compare the risks of $\hatbX^{(l)}$ and $\widehat{\bX}^{(\sdec)}$. We have 
\[\risk_n(\hatbX^{(l)};f)  - \risk_n(\widehat{\bX}^{(\sdec)};f) = -\frac{2}{n^2} \E\big\langle \hatbX^{(l)} - \widehat{\bX}^{(\sdec)}, \bX_f\big\rangle  +\frac{1}{n^2} \big(\E \|\hatbX^{(l)}\|_F^2 - \E \|\widehat{\bX}^{(\sdec)}\|_F^2\big).\] 
 We have
 \begin{align*}
 \E\big\langle \hatbX^{(l)} - \widehat{\bX}^{(\sdec)}, \bX_f\big\rangle &= 
 \sum_{u,v \in V_n} \E \Big[ \Big(\E\big[f(\theta_u)|Y^{(\eps)}_{B_{G_n}(u,l)}\big] \E\big[f(\theta_v)|Y^{(\eps)}_{B_{G_n}(v,l)}\big]\\
 & \hspace{2cm}- \E\big[f(\theta_u)|Y^{(\eps)}_{G_n}\big] \E\big[f(\theta_v)|Y^{(\eps)}_{G_n}\big]\Big) f(\theta_u)f(\theta_v) \Big].
 \end{align*}
 By consecutive triangle inequalities, this is bounded in absolute value by
\begin{align*} 
\|f\|_{\infty}^2 \sum_{u,v \in V_n} &\E \Big[ \Big|\E\big[f(\theta_u)|Y^{(\eps)}_{B_{G_n}(u,l)}\big] \E\big[f(\theta_v)|Y^{(\eps)}_{B_{G_n}(v,l)}\big]
- \E\big[f(\theta_u)|Y^{(\eps)}_{G_n}\big] \E\big[f(\theta_v)|Y^{(\eps)}_{G_n}\big]\Big| \Big]\\
 &\le 2n\|f\|_{\infty}^3 \sum_{u \in V_n} \E \Big[ \Big|\E\big[f(\theta_u)|Y^{(\eps)}_{B_{G_n}(u,l)}\big] 
- \E\big[f(\theta_u)|Y^{(\eps)}_{G_n}\big]\Big|\Big]\\
&\le 2n\|f\|_{\infty}^4 \sum_{u \in V_n} \sum_{x \in \cX} \E \big[ \big|\widehat{\mu}_{G_n,u,l}(x) -\mu_{G_n,u}(x)\big|\big] \\
&=  4n\|f\|_{\infty}^4 \sum_{u \in V_n} \E \big[ d_{\sTV}(\widehat{\mu}_{G_n,u,l} ,\mu_{G_n,u})\big].
\end{align*}
Here, $\|f\|_{\infty}$ denotes the supremum norm of $f$. 

On the other hand, and following a similar strategy,
\begin{align*}
\E \big\|\hatbX^{(l)}\big\|_F^2 - \E \big\|\widehat{\bX}^{(\sdec)}\big\|_F^2
&\le 2n \|f\|_{\infty}^3  \sum_{u \in V_n} \E\Big[\Big|\E\big[f(\theta_u) |Y^{(\eps)}_{B_{G_n}(u,l)}\Big]- \E\big[f(\theta_u) | Y^{(\eps)}_{G_n}\big]\Big|\Big]\\
& \le 4n\|f\|_{\infty}^4 \sum_{u \in V_n} \E \big[d_{\sTV}(\widehat{\mu}_{G_n,u,l},\mu_{G_n,u})\big].
\end{align*}
Invoking Theorem~\ref{main} concludes the proof.  
\end{proofof}

Theorem \ref{main} and Proposition \ref{L2_statement} allow to control other metrics for the estimation errors beyond $\risk_n(\hatbX;f)$.
As an example, we consider  the `overlap' metric that applies to estimators $\hbtheta: \cY^{E_n}\times \cX_{\star}^{V_n}\to \cX^{V_n}$ 
which assign labels to vertices.
We define      
\begin{equation}\label{overlap}
\overlap(\hbtheta,\btheta) := \max_{\sigma \in \mathscr{S}_{q}} \frac{1}{|V_n|}\sum_{u \in V_n} \bfone\big\{\hat{\theta}_u=\sigma(\theta_{u})\big\},
\end{equation}
where $\mathscr{S}_{q}$ is the set of permutations on $\cX$, with $q =
|\cX|$. 

As a corollary of Proposition~\ref{L2_statement}, the overlap between a sample from the local marginals and $\btheta$ can be lower-bounded in a nontrivial way (the proof can be found in Appendix~\ref{sec:omitted}):  
\begin{corollary}\label{overlap_amenable}
For each let $l \ge 1$, let $\hat{\btheta}^{(l)} = (\hat{\theta}^{(l)}_u)_{u \in V_n}$ where $\hat{\theta}^{(l)}_u \sim \widehat{\mu}_{G_n,u,l}$ independently for all $u \in V_n$.  Then  for
 almost every $\eps>0$,
\[\liminf_{l \to \infty} \lim_{n \to \infty} \E\big[\overlap\big(\hbtheta^{(l)},\btheta\big)\big] \ge  \sum_{x \in \cX}\E\big[\mu^2_{G,o}(x)\big].\]
\end{corollary}
As the radius $l$ of the local balls increases, the performance of $\hbtheta^{(l)}$ approaches that of a sample drawn from the full marginals $(\mu_{G_n,u})_{u \in V_n}$. 

\section{Results for random regular graphs}
\label{sec:RandRegular}

The assumption of anchored--amenability is crucial in the proofs of Theorems \ref{thm:Risk} and
\ref{main}. While we do not know whether a weaker condition is sufficient, we show that these results do not hold for at least one non-amenable case, namely, when 
$G_n$ is a random $k$-regular graph with constant degree $k$. For the case of $\integers_q$--synchronization we show that in a certain regime of signal-to-noise ratio (SNR),
 the local estimates of vertex marginals provide no information about the hidden assignment $\btheta$, while in the same regime, it is information-theoretically possible to estimate $\btheta$ non-trivially. 

As mentioned in the introduction, an information-computation gap has been observed in several statistical
models. However, none of the rigorous results in the literature matches the setting of Theorems \ref{thm:Risk} and
\ref{main}. To the best of our knowledge, the closest example  is the case of the stochastic block model with $q$ communities on sparse random graphs (see \cite{abbe2017community} for a comprehensive survey and references therein). 
As explained in Section \ref{sec:Examples}, this example fits our framework, although with $G_n$ being the complete graph. In particular, $G_n$ does not
converge to a locally finite graphs. In contrast, the example treated in this section satisfies all the assumptions of Theorems \ref{thm:Risk} and \ref{main} except amenability (and tameness).
Proofs for this section are deferred to Appendices~\ref{sec:IT-Zq} and~\ref{sec:Loc-Zq}.

\subsection{Information-theoretic reconstruction: An exhaustive search algorithm}     

Given a graph $G = (V,E)$ on $n$ vertices, $\btheta\in \cX^{V}$ and $\bY\in \cY^{E}$, we define the edge empirical distribution
\begin{align}
\hnu^G_{\btheta,\bY} := \frac{1}{|E|}\sum_{(u,v)\in E}\delta_{(\theta_{u},\theta_{v},Y_{uv})}\, .
\end{align}
This is a probability distribution on $\cX\times\cX\times \cY$: $\hnu^{G}_{\btheta,\bY}\in\cuP(\cX\times\cX\times \cY)$. 
(Recall that $\cuP(S)$ denotes the simplex of probability distributions over the set $S$.) 
Define $\onu\in\cuP(\cX)$ to be the uniform distribution on $\cX$ and $\onu_{\se}\in \cuP(\cX\times\cX\times \cY)$ via
\begin{align*}
\onu_{\se} (\theta_1,\theta_2,y_{12}) = \onu(\theta_1)\, \onu(\theta_2)\, Q(y_{12}|\theta_1,\theta_2)\, .
\end{align*}
We then define the set of `typical' assignments of node variables by
\begin{align*}
\Theta(\eta;G,\bY) := \Big\{\btheta\in \cX^V:\;\; d_{\sTV}(\hnu^G_{\btheta,\bY},\onu_{\se})\le \eta\Big\}\, .
\end{align*}
We then consider the reconstruction algorithm that outputs a typical configuration 
\begin{align}
\hbtheta(G,\bY)\in \Theta(\eta_n;G,\bY)\, ,\qquad \eta_n:= \frac{(\log n)}{\sqrt{n}}\, . \label{eq:IT-Estimator}
\end{align}
If $\Theta(\eta_n;G,\bY)$ is empty, we define $\hbtheta(G,\bY)$ arbitrarily (for instance $\hbtheta(G,\bY) = \btheta_*$ for a fixed reference configuration $\btheta_*\in\cX^V$).
If $\Theta(\eta_n;G,\bY)$ contains more than one element, then $\hbtheta(G,\bY)$ selects one arbitrarily,
e.g., the first one in lexicographic order. In fact our proofs apply to any algorithm that satisfy condition 
\eqref{eq:IT-Estimator} with high probability. As discussed below (see Remark \ref{rmk:TwoEstimators}) this condition is also satisfied by the randomized estimator
$\hbtheta\sim \P\big(\cdot | Y^{(\eps)}_{G_n}\big)$ that samples from the posterior.

It is immediate to show that the typical set is non-empty with high probability. (Throughout this section, we use $\btheta_0$ for the ground truth, in order to
distinguish it from a generic vector $\btheta\in\cX^n$.)
\begin{lemma}\label{lemma:TypicalNonEmpty}
Let $G_n$ be a random $k$-regular graph on $n$ vertices, and let $(\btheta_0,\bY)$ be distributed according to the random observation model described in 
the Introduction. Then, there exists $c_0 = c_0(|\cX|,|\cY|)>0$ such that 
\begin{align*}
\prob\Big(\btheta_0 \in \Theta(\eta_n;G_n,\bY) \Big)\ge 1- c_0^{-1}\exp\big\{-c_0(\log n)^2\big\}\, .
\end{align*}
\end{lemma}

\begin{remark}\label{rmk:TwoEstimators}
As mentioned above, one might consider a randomized estimator $\hbtheta$  that outputs a sample from the posterior: $\hat{\btheta} \sim \P\big(\,\cdot\, | Y^{(\eps)}_{G_n}\big)$. 
Note that this satisfies the condition $\hat{\btheta} \in \Theta(\eta_n,G_n,\bY)$ (cf. Eq. \eqref{eq:IT-Estimator}) 
with the same probability $1- c_0^{-1}\exp\{-c_0(\log n)^2\}$. Indeed this follows simply by noting that, with this definition, the pair $(\hbtheta,\bY)$
is distributed as $(\btheta_0,\bY)$.
Therefore all the results to follow apply to this randomized estimator as well. 
\end{remark}
Given two assignments $\btheta_0,\btheta\in \cX^V$, we define their joint empirical vertex distribution as
\begin{align}\label{overlap_dist}
\homega_{\btheta_0,\btheta} := \frac{1}{|V|}\sum_{u\in V}\delta_{\theta_{0,u},\theta_u}\, .
\end{align}
This is a probability distribution on $\cX\times\cX$:  $\homega_{\btheta_0,\btheta}\in \cuP(\cX\times\cX)$.

To state the next result let us briefly recall some notions form information theory. Given a discrete random variable (or random vector) $X$, we denote by $H(X)$ the Shannon entropy of the law of $X$, namely --with a slight abuse of notation--  $H(X) = H(P_X)=-\sum_xP_X(x)\log P_X(x)$. For a vector $(X_1,\dots,X_m)$, $H(X_1,\dots,X_m) = H(P_{X_1,\dots,X_m})$. 
The conditional entropy is defined by $H(X|Y) = H(X,Y)-H(Y)$, and the mutual information by $I(X;Y) = H(X)-H(X|Y) = H(Y)-H(Y|X)$. 
\begin{customthm}{C}\label{exhaustive_general}
Assume there exists $c_M>0$ such that $c_M^{-1}\le Q(y|x_1,x_2)\le c_M$ for all $x_1,x_2\in\cX$, $y\in\cY$,
and let $(\theta_1,\theta_2,Y)$ have joint distribution $\onu_e$ (recall that $\onu_e(x_1,x_2,y) = \onu(x_1)\onu(x_2) \times Q(y|x_1,x_2)$ where $\onu$ is the uniform distribution over $\cX$). If 
\begin{align*}
\frac{k}{2}I(\theta_1,\theta_2;Y) \ge H(\theta_1)+\eps\, ,
\end{align*}
for some $\eps>0$, then  there exists $\delta = \delta(\eps,c_M)>0$ and a constant $c_0>0$ such that
\begin{align*}
\prob\Big(d _{\sTV} (\homega_{\hbtheta,\btheta_0},\onu\times\onu)\ge \delta\Big) \ge 1- c_0^{-1}\exp\{-c_0(\log n)^2\}\, .
\end{align*}
\end{customthm}
The proof of this theorem relies on a truncated first moment method, where we count the expected number of typical assignments $\btheta \in \Theta(\eta_n;G_n,\bY)$ having a given value of the empirical overlap distribution $\homega_{\btheta,\btheta_0}$, conditioned on certain typicality constraints on the instance $(G_n,\btheta_0,\bY)$. The full argument is deferred to Appendix \ref{sec:ProofExhaustiveGeneral}. (We refer, e.g., to \cite{dembo2013factor} for similar calculations in a somewhat simpler context.)

The next corollary applies the result of Theorem~\ref{exhaustive_general} to $\Z_q$--synchronization.
\begin{corollary}\label{z_q_corollary}
Consider the $\Z_q$--synchronization problem. If 
\begin{align*}
k>k_*(p;q)   :=\frac{2\log q}{\Big(1-p+\frac{p}{q}\Big)\log\big(p+q(1-p)\big) +\Big(1-\frac{1}{q}\Big)p\log p} \,  ,
\end{align*}
then there exists $\delta, c_0>0$ depending on $k,p,q$ such that, with probability at least $1-c_0^{-1}\exp\{-c_0 (\log n)^2\}$,
$d_{\sTV}(\homega_{\hbtheta,\btheta_0},\onu\times\onu)\ge \delta$. 

Furthermore, as $p\to 1$, we have
\begin{align*}
k_*(p;q) = \frac{4\log q}{(q-1)(1-p)^2} +\bigo\left((1-p)^{-1}\right)\, .
\end{align*}
\end{corollary}
This corollary follows from Theorem~\ref{exhaustive_general} simply by computing $I(\theta_1,\theta_2;Y)$ in the case of $\Z_q$--synchronization. We omit the details. 
Finally, we deduce from Theorem~\ref{exhaustive_general} the possibility of weak recovery.
\begin{corollary} \label{overlap_corollary}
Under the assumptions of Theorem~\ref{exhaustive_general}, if $\frac{k}{2}I(\theta_1,\theta_2;Y) \ge H(\theta_1)+\eps$, then there exists a constant $\delta  = \delta(\eps)>0$ such that 
\begin{align}
\liminf_{n \to \infty} \E[\overlap(\hbtheta,\btheta_0)] \ge  \frac{1}{q} + \delta. \label{eq:Rec-1}
\end{align}
Moreover, there exists a function $f : \cX \mapsto \R$ with zero mean, unit variance, and a constant $\delta  = \delta(\eps,|\cX|,c_M)>0$ such that 
\begin{align}
\limsup_{n \to \infty} \risk_n^{\textup{Bayes}}(f) \le 1 - \delta. \label{eq:Rec-2}
\end{align}
In particular, the conclusions  \eqref{eq:Rec-1} and \eqref{eq:Rec-2} hold in the $\Z_q$--synchronization model if $k>k_*(p;q)$.
\end{corollary}

\subsection{Performance of the local algorithm}
 In this section we examine the asymptotics of the local marginals 
 \[\widehat{\mu}_{G_n,u,l}(x) = \P\Big(\theta_u = x | Y^{(\eps)}_{B_{G_n}(u,l)}\Big),\]
 when $G_n$ is a random $k$-regular graph, in the special case of $\Z_q$--synchronization with side information from \BEC($\epsb$). 

 We have seen in the previous section that weak recovery is possible (albeit non-efficiently) when $(1-p)^2 k > \frac{4 \log q}{q-1} + \bigo(1-p)$ even in the absence of side information (Corollary~\ref{overlap_corollary}). We show on the other hand that the local marginals are approximately uniform if $(1-p)^2 (k-1) < 1$. The latter condition is known as the \emph{Kesten-Stigum threshold} for the problem of robust reconstruction on the tree~\cite{janson2004robust}.     
 
 \begin{customthm}{D}\label{local_alg_Z_q}
 Consider $\Z_q$--synchronization with side information from \BEC($\epsb$) on a random $k$-regular graph $G_n$. There exist constants $c = c(k,p,q)$ and $C = C(k,p,q)$ such that the following holds. If $(1-p)^2 (k-1) < 1$ and $\eps \le c$ then   
 \begin{equation}\label{local_on_tree}
 \limsup_{l \to \infty}\limsup_{n \to \infty} \frac{1}{n}\sum_{u\in V_n}\E \big[d_{\sTV}(\widehat{\mu}_{G_n,u,l},\onu)^2\big] \le C \eps.
 \end{equation}
 \end{customthm}
 
The above theorem implies that all estimators $(\hat{\btheta}^{(l)})_{l \ge 1}$ where $\hat{\theta}^{(l)}_u \sim \widehat{\mu}_{u,l}$ independently for all $u \in V_n$, have almost trivial performance. Recall the definition of the matrix $\hatbX^{(l)}$:
\[\widehat{X}^{(l)}_{uv} = \E\Big[f(\theta_u)\big|Y^{(\eps)}_{B_{G_n}(u,l)}\Big]\cdot \E\Big[f(\theta_v)\big|Y^{(\eps)}_{B_{G_n}(v,l)}\Big],~~~ u , v \in V_n.\]
\begin{corollary}\label{coro:Loc-Zq}
In the setting of Theorem~\ref{local_alg_Z_q}, if $(1-p)^2 (k-1) < 1$, then there exists constants $c_1 ,c_2>0$ depending on $k$ and $q$ such that  
\[\limsup_{l \to \infty} \limsup_{n \to \infty} \E\big[\overlap\big(\hat{\btheta}^{(l)},\btheta_0\big)\big] \le \frac{1}{q} + c_1\sqrt{\eps}.\]
Moreover, for all $ f : \Z_q \mapsto \R$ with zero mean and unit variance,  
\[\liminf_{l \to \infty} \liminf_{n \to \infty}  \risk_n\big(\hatbX^{(l)};f\big) \ge 1 - c_2\|f\|_{\infty}^2\eps.\]
\end{corollary}

\begin{remark}
The above implies that no local algorithm can estimate $\bX_f$ with non-trivial accuracy. Indeed, the estimator of $f(\theta_u)f(\theta_v)$ of minimal risk based on the information contained is the balls of radius $l$ centered around $u$ and $v$ respectively is $\E\big[f(\theta_u)f(\theta_v) | Y^{(\eps)}_{B_{G_n}(u,l) \cup B_{G_n}(v,l)}\big]$. The latter quantity is equal to $\widehat{X}^{(l)}_{uv}$ if the two balls are disjoint, which is the case for $1-o_n(1)$ fraction of pairs of vertices $(u,v)$ when $l$ is held constant.    
\end{remark}

The proof of Theorem~\ref{local_alg_Z_q} is deferred to Appendix \ref{sec:Proof_local_alg_Z_q}, but we give here an outline.
We use  local--weak convergence to first lift the problem to the infinite $k$-regular tree, in which the study of the local marginals reduces to the study of a certain distributional recursion. Then we prove that below the Kesten-Stigum threshold, the uniform distribution $\onu$ is a stable fixed point of this recursion. The argument proceeds as follows.
 Let $o$ be the root of infinite $(k-1)$--ary tree $T$ and denote by $T(l)$ the subtree consisting of the first $l$ generations of $T$ rooted at $o$. Now let $\mu_{o,l}(x) :=\P\big(\theta_o = x | \bY^{(\eps)}_{T_k(l)}\big)$ for all $x \in \Z_q$ and consider the sequence $z_l :=  \E[\mu_{o,l} (\theta_o)| \xi_o = \star] - \frac{1}{q}$ which measures the deviation from uniformity of the local marginal at the root. We use the recursive structure of the tree to show that for $\eps$ small enough and $\kappa = (1-p)^2(k-1)$, the sequence $(z_l)_{l \ge 0}$ satisfies the approximate recursion
\begin{equation}\label{recursion}
\big|z_{l+1} - (1-\eps)\kappa z_l - \eps\kappa \frac{q-1}{q}\big| \le C(q) \kappa^2 \big( z_l^2+ \eps^2\big),
\end{equation}
where $C(q)$ is constant depending only on $q$. Since $z_0 = 0$, this implies that if $\kappa <1$ then the sequence stays within an interval of size  $C'(q,\kappa)\eps$ around the origin. This, in turn, can be converted to the claim of Theorem~\ref{local_alg_Z_q}. The analysis of this recursion originates in the study of the 
robust reconstruction problem on the tree. In this problem, a spin at the root (an $\cX$-valued r.v.) is broadcast through noisy channels along edges of the tree.
The statistician observes a noisy realization of this process on the leaves of $T(l)$ for large $l$, and is tasked with inferring the value at the root (see e.g., \cite{evans2000broadcasting,mossel2003information,janson2004robust}). Similar recursions also arise in the study of the `robustness' of phase transitions in the Ising model on the tree~\cite{pemantle1999robust}. In particular, our analysis builds on ideas from~\cite{mezard2006reconstruction,sly2011}.

\section{Proof of Proposition~\ref{L2_statement}}
\label{sec:Proof_L2Statement} 

We start with the proof of~\eqref{local}, which is straightforward and does not need the amenability assumption. The proof of~\eqref{global} will crucially hinge upon a property of decay of certain point--to--set correlations (Lemma~\ref{triviality}), which we establish using anchored--amenability and the presence of $\eps$--side information.   
 For ease of notation, we adopt the following convention in this section: in quantities of the form  $\P\big(\theta_o = x | Y^{(\eps)}_{A}\big)$ where $A$ is any subgraph of $G$, it is implicit that the rooted graph $(G,o)$ is also conditioned on, abbreviating the more accurate but lengthier notation $\P\big(\theta_o = x | (G,o),Y^{(\eps)}_{A}\big)$.
 
\subsection{Proof of the `local' statement~\eqref{local}}
Let $x \in \cX$ and $l \ge 1$. The function $f:\mathcal{G}_* \to [0,1]$ defined by $f(G,o) = \E\big[\P\big(\theta_o = x | Y^{(\eps)}_{[G,o]_l}\big)^2\big]$ is clearly continuous in the topology of local  convergence. Indeed for $(G_n,o_n) \xrightarrow[]{loc.} (G,o)$, let $n_0 \ge 1$ such that $[G_n,o_n]_l \equiv [G,o]_l$ for all $n \ge n_0$. Hence $f(G_n,o_n) = f(G,o)$ for all $n \ge n_0$.  
Since $f$ is also bounded, we obtain by local--weak convergence under uniform rooting that 
\begin{align*}
\frac{1}{|V_n|} \sum_{u \in V_n} \E\big[\P\big(\theta_u = x | Y^{(\eps)}_{B_{G_n}(u,l)}\big)^2\big] &=
\frac{1}{|V_n|} \sum_{u \in V_n} \E\big[\P\big(\theta_u = x | Y^{(\eps)}_{[G_n,u]_l}\big)^2\big] \\
&= \E_{\rho_{G_n}} \big[f(G_n,o_n)\big]\\
&\xrightarrow[n\to \infty]{}   \E_{\rho} \big[f(G,o)\big]\\
&= \E\Big[\P\big(\theta_o = x | Y^{(\eps)}_{[G,o]_l}\big)^2\Big].
\end{align*}
Next, we observe that the sequence $\big(\P\big(\theta_o = x | (G,o),
Y^{(\eps)}_{[G,o]_l}\big)\big)_{l \ge 1}$ is a bounded martingale,
therefore it converges almost surely and in 
$\mathbb{L}_2$ to $\P\big(\theta_o = x | Y^{(\eps)}_{G}\big)$ by L\'evy's upward  theorem. This concludes the proof of the first statement~\eqref{local}:
\[\lim_{l\to \infty} \lim_{n\to \infty}\frac{1}{|V_n|} \sum_{u \in V_n} \E\big[\P\big(\theta_u = x | Y^{(\eps)}_{B_{G_n}(u,l)}\big)^2\big]  = \E\Big[\P\big(\theta_o = x | Y^{(\eps)}_{G}\big)^2\Big].\]      

\subsection{Proof of the `global' statement~\eqref{global}}
The proof breaks into three parts. First, we easily obtain a lower bound from Jensen's inequality: 
\[\E\Big[\P\big(\theta_u = x | Y^{(\eps)}_{G_n}\big)^2\Big] = \E\Big[\E\big[\P\big(\theta_u = x | Y^{(\eps)}_{G_n}\big)^2 | Y^{(\eps)}_{B_{G_n}(u,l)}\big]\Big] \ge \E\Big[\P\big(\theta_u = x | Y^{(\eps)}_{B_{G_n}(u,l)}\big)^2\Big].\]  
Therefore
\begin{align}
\liminf_{n \to \infty} \frac{1}{|V_n|} \sum_{u \in V_n} \E\big[\mu_{G_n,u}^2(x)\big] &\ge  \lim_{l \to \infty} \lim_{n \to \infty} \frac{1}{|V_n|} \sum_{u \in V_n} \E\Big[\P\big(\theta_u = x | (G,o), Y^{(\eps)}_{B_{G}(u,l)}\big)^2\Big] \nonumber\\
&= \E\big[\mu^2_{G,o}(x)\big],\label{lower_bound}
\end{align}
where the last equality is the content of statement~\eqref{local}. As for the upper bound, we have 
\begin{lemma}
Consider the $\sigma$-algebra 
\[\mathcal{T}^{(\eps)}_{\infty} = \bigcap_{l \ge 1}\sigma\Big(\big\{Y_{uv} : (u,v)\in E\big\}\cup\big\{\xi^{(\eps)}_u: u \in V\big\}\cup\big\{ \theta_u : u \in V : d_G(o,u) \ge l\big\}\Big),\] where  $d_G$ is the distance in $G$. Then
\begin{equation}\label{upper_bound}
\limsup_{n \to \infty} \frac{1}{|V_n|} \sum_{u \in V_n} \E\big[\mu_{G_n,u}^2(x)\big] \le 
\E\Big[\P\big(\theta_o = x | (G,o),\mathcal{T}^{(\eps)}_{\infty}\big)^2\Big].
\end{equation}  
\end{lemma}    

\begin{proof}
Fix $u \in V_n$ and $x \in \cX$. We condition on the r.v.'s $\theta_{\partial B_{G_n}(u,l)} := \{\theta_{v} : v \in V_n, d_{G_n}(u,v)=l\}$ and apply Jensen's inequality: 
\[\E\Big[\P\big(\theta_u = x | Y^{(\eps)}_{G_n}\big)^2\Big] \le \E\Big[\P\big(\theta_u = x | Y^{(\eps)}_{G_n}, \theta_{\partial B_{G_n}(u,l)}\big)^2\Big].\]
We now observe that conditionally on the boundary variables $\theta_{\partial B_{G_n}(u,l)}$, $\theta_u$ is independent of $Y^{(\eps)}_{vw}$ for all $v$ and $w$ outside the ball $B_{G_n}(u,l)$. This is guaranteed by the spatial Markov property of the model. Therefore 
\[\E\Big[\P\big(\theta_u = x | Y^{(\eps)}_{G_n}, \theta_{\partial B_{G_n}(u,l)}\big)^2\Big] = \E\Big[\P\big(\theta_u = x | Y^{(\eps)}_{B_{G_n}(u,l)}, \theta_{\partial B_{G_n}(u,l)}\big)^2\Big].\]
The event on the right--hand side is localized to a ball of fixed radius. So by local--weak convergence, we pass to the limiting rooted graph $(G,o)$, (similarly to the proof of~\eqref{local}):
\[\limsup_{n \to \infty} \frac{1}{|V_n|} \sum_{u \in V_n} \E\big[\mu_{G_n,u}^2(x)\big] \le  \E\Big[\P\big(\theta_o = x | (G,o), Y^{(\eps)}_{B_G(o,l)}, \theta_{\partial B_G(o,l)}\big)^2\Big].\]   
Now using the same Markov property as above, the expectation in the right--hand side remains unchanged if we further condition on $\cF_{o}^{ \ge l} := \{\theta_{v} : v \in B_G(o,l)^c\}$ and $\{Y^{(\eps)}_{u,v} : u,v \in B_G(o,l)^c\}$, which are beyond the boundary of $B_G(o,l)$: the extra information is irrelevant to $\theta_o$. We arrive at the upper bound
\[\sum_{x \in \cX} \E\Big[\P\big(\theta_o = x | (G,o), Y^{(\eps)}_{G}, \cF_{o}^{ \ge l}\big)^2\Big].\] 
Now we observe that the sequence $\big(\P\big(\theta_o = x | (G,o), Y^{(\eps)}_{G}, \cF_{o}^{ \ge l}\big)\big)_{l\ge 1}$ is a bounded \emph{backward} martingale (since the corresponding filtration is decreasing), which converges to $\P\big(\theta_o = x | (G,o),\mathcal{T}^{(\eps)}_{\infty}\big)$ a.s.\ and in $\mathbb{L}_2$ by L\'evy's downward theorem. 
This concludes the argument. 
\end{proof}%

The last piece of the proof is to show that the lower and upper bounds~\eqref{lower_bound} and~\eqref{upper_bound} coincide when $(G,o)$ is a.s.\ anchored--amenable:
\begin{proposition}\label{triviality}
Assume $(G,o)\sim \rho$ is unimodular, almost surely anchored--amenable and tame. Then for almost every $\eps>0$ and all $x \in \cX$,
 \[\P\Big(\theta_o = x \Big| (G,o), \mathcal{T}^{(\eps)}_{\infty}\Big) = \P\Big(\theta_o = x \Big| (G,o), Y^{(\eps)}_{G}\Big) \qquad \mbox{a.s.}\]
\end{proposition}

This is the only part of the proof which requires assumptions of the limiting random rooted graph, and the presence of non-zero side information from \BEC($\epsb$). We reiterate that unimodularity is guaranteed if $(G,o)$ is the limit of a sequence of finite graphs (see Section~\ref{sec:Background}), so it is automatically satisfied in our setting.

The first ingredient in the proof of Proposition~\ref{triviality} is the following generic lemma that allow to control the dependency under the posterior between variables $\theta_u$ associated to
vertices in the interior of a set $S \subset V(G)$ and variables $\theta_{\partial S}$ associated to the boundary of this set.
Our first lemma bounds the mutual information between $\theta_u$ and $\theta_{\partial S}$. This result is inspired by Lemma 3.1 in~\cite{MontanariSparse}.
Let us recall the definition of conditional mutual information between $X$ and $Y$ given $Z$: $I(X;Y|Z) = H(X|Z)-H(X|Y,Z) = H(Y|Z)-H(Y|X,Z)$, where $H(X|Y) = H(X,Y)-H(Y)$ is the conditional entropy.
\begin{lemma}\label{mutual_info}
Let $G$ be a graph, and $S \subset V(G)$ finite and non-empty. For all $\eps \ge 0$, we have
\[\sum_{u \in S} \int_0^{\eps} I\Big(\theta_u ; \theta_{\partial S} \Big| Y^{(\eps')}_{S}\Big) \rmd \eps' \le \log |\cX| \cdot |\partial S|.\]
\end{lemma}
\begin{proof}
The argument relies on differentiating the conditional Shannon entropy of $\theta_{\partial S}$ given $Y^{(\eps)}_{S}$ with respect to $\eps$. Let us first replace the single parameter $\eps$ (the probability of non-erasure) by a set of parameters $\underline{\eps} = (\eps_u)_{u \in S}$: for each vertex $u$, $\theta_u$ is revealed with probability $\eps_u$. We also replace the notation $Y^{(\underline{\eps})}_{S}$ by $(Y, \xi)$, omitting an explicit reference to $\underline{\eps}$ and to the ball $S$. We finally denote $\xi^{\backslash (u)} = \{\xi_v : v \in S, v \neq u\}$ with $\xi_u$ removed. 
We have 
\[H(\theta_{\partial S} | Y,\xi) = \eps_u H\big(\theta_{\partial S} | Y,\xi^{\backslash (u)}, \theta_u \big) + (1-\eps_u)H\big(\theta_{\partial S} | Y,\xi^{\backslash (u)}\big).\]
Taking a derivative w.r.t.\ $\eps_u$ yields:
\begin{align*}
\frac{\rmd}{\rmd \eps_u}H(\theta_{\partial S} | Y,Z) &=  H\big(\theta_{\partial S} | Y,\xi^{\backslash (u)}, \theta_u \big) - H\big(\theta_{\partial S} | Y,\xi^{\backslash (u)}\big)\\
&= - I\big(\theta_u; \theta_{\partial S} | Y,\xi^{\backslash (u)}\big),
\end{align*}
 where the latter is the conditional mutual information of $\theta_u$ and $\theta_{\partial S}$ given $(Y,\xi^{\backslash (u)})$. Now we set $\eps_u = \eps$ for all $u \in S$. We obtain 
 \[\frac{\rmd}{\rmd \eps}H(\theta_{\partial S} | Y,\xi) = -\sum_{u \in S}  I\big(\theta_u; \theta_{\partial S} | Y,\xi^{\backslash (u)}\big).\]
 We now integrate w.r.t.\ $\eps$:
 \begin{align*}
 \int_0^{\eps} \sum_{u \in S}  I\big(\theta_u; \theta_{\partial S} | Y,\xi^{\backslash (u)}\big) \rmd \eps' &=H(\theta_{\partial S} | Y,\xi^{(\eps = 0)}) - H(\theta_{\partial S} | Y,\xi^{(\eps)})\\
 &\le H(\theta_{\partial S} | Y,\xi^{(\eps = 0)})\\
 &\le  H(\theta_{\partial S})\\
 &\le \sum_{u \in \partial S} H(\theta_u)\\
 &= \log |\cX| \cdot | \partial S|.
 \end{align*}
The second line is by positivity of entropy, the third line follows from the fact that conditioning reduces the entropy, the fourth line is by sub-additivity, and the last line is since $\theta_u$ is marginally uniform on $\cX$.   
Now we finish the proof by observing that $I\big(\theta_u; \theta_{\partial S} | Y,\xi\big) = I\big(\theta_u; \theta_{\partial S} | Y,\xi^{\backslash (u)}, \xi_u\big) \le I\big(\theta_u; \theta_{\partial S} | Y,\xi^{\backslash (u)}\big)$ because the left--hand side vanishes whenever $\xi_u \neq \star$.
\end{proof}

Next, we translate Lemma~\ref{mutual_info} into an average statement about decay to point--to--set correlations:
\begin{lemma}\label{lem:point_to_set}
Let $G$ be a graph and $S \subset V(G)$ finite and non-empty. Then for all $\eps>0$,
\[ \frac{1}{|S|}\sum_{u \in  S} \int_{0}^{\eps} \sum_{x \in \cX}\E\Big[\Big(\P\big(\theta_u = x | \mathcal{T}^{(\eps')}_{\infty}\big) - \P\big(\theta_u = x | Y^{(\eps')}_{G}\big)\Big)^2\Big]  \rmd \eps' \le  \sqrt{8\eps \log |\cX|  \cdot |\partial S| / |S|}.\] 
\end{lemma}
\begin{proof}
Let $L \ge \text{diam}(S)$ so that $S \subseteq B_G(u,L)$ for every $u\in S$. By Jensen's inequality we have
\[ \E\Big[\P\Big(\theta_u = x \big| Y^{(\eps)}_{B(u,L)}, \theta_{\partial B(u,L)}\Big)^2\Big] 
\le
\E\Big[\P\Big(\theta_u = x \big| Y^{(\eps)}_{S}, \theta_{\partial S}\Big)^2\Big],\]
and
\[\E\Big[\P\Big(\theta_u = x \big| Y^{(\eps)}_{B(u,L)}\Big)^2\Big]  
\ge 
\E\Big[\P\Big(\theta_u = x  \big| Y^{(\eps)}_{S}\Big)^2\Big].\]
Therefore,
\begin{align}\label{ineq:first} 
&\E\Big[\Big( \P\big(\theta_u = x \big| Y^{(\eps)}_{B(u,L)}, \theta_{\partial B(u,L)}\big) 
 - 
\P\big(\theta_u = x \big| Y^{(\eps)}_{B(u,L)}\big)\Big)^2\Big]    \\
 &\qquad \qquad=
 \E\Big[\P\big(\theta_u = x \big| Y^{(\eps)}_{B(u,L)}, \theta_{\partial B(u,L)}\big)^2\Big] 
 - 
 \E\Big[\P\big(\theta_u = x \big| Y^{(\eps)}_{B(u,L)}\big)^2\Big]  \nonumber \\
&\qquad \qquad\le
\E\Big[\P\big(\theta_u = x \big| Y^{(\eps)}_{S}, \theta_{\partial S}\big)^2\Big]
- 
\E\Big[\P\big(\theta_u = x  \big| Y^{(\eps)}_{S}\big)^2\Big] \nonumber \\
&\qquad \qquad = \E\Big[\Big(\P\big(\theta_u = x \big| Y^{(\eps)}_{S}, \theta_{\partial S}\big) - \P\big(\theta_u = x  \big| Y^{(\eps)}_{S}\big)\Big) ^2\Big] \nonumber\\
&\qquad \qquad\le 2  \E\Big[\Big|\P\big(\theta_u = x \big| Y^{(\eps)}_{S}, \theta_{\partial S}\big) - \P\big(\theta_u = x  \big| Y^{(\eps)}_{S}\big)\Big| \Big].\nonumber
\end{align}
The last quantity is equal to
\begin{align*}
&  2\E\Big[ \sum_{\sigma \in \cX^{\partial S}}\P\big(\theta_{\partial S} = \sigma \big| Y^{(\eps)}_{S}\big) \cdot \Big| \P\big(\theta_u = x \big| Y^{(\eps)}_{S}, \theta_{\partial S} = \sigma \big) - \P\big(\theta_u = x \big| Y^{(\eps)}_{S}\big)\Big|\Big]\\
&\qquad = 2\sum_{\sigma \in \cX^{\partial S}} \E\Big| \P\Big(\theta_u = x, \theta_{\partial S} = \sigma  \big| Y^{(\eps)}_{S}\Big) - \P\Big(\theta_u = x \big| Y^{(\eps)}_{S}\Big) \cdot \P\Big( \theta_{\partial S} = \sigma\big| Y^{(\eps)}_{S}\Big)\Big|. 
\end{align*}
Summing over $x$ and using~\eqref{ineq:first} we get
\begin{align*}
 \sum_{x \in \cX} \E\Big[\Big( &\P\big(\theta_u = x \big| Y^{(\eps)}_{B(u,L)}, \theta_{\partial B(u,L)}\big) 
 - 
\P\big(\theta_u = x \big| Y^{(\eps)}_{B(u,L)}\big)\Big)^2\Big]   \nonumber \\
& \le 4\E d_{\sTV}\Big( \P\big((\theta_u , \theta_{\partial S}) \in \cdot  | Y^{(\eps)}_{S}\big) ,~ \P\big(\theta_u \in \cdot  | Y^{(\eps)}_{S}\big) \times \P\big( \theta_{\partial S} \in \cdot | Y^{(\eps)}_{S}\big) \Big).
\end{align*}
We send $L$ to infinity and use martingale convergence on the left-hand side, and use Pinsker's inequality on the right-hand side to obtain  
\begin{align*}
 \sum_{x \in \cX}\E\Big[\Big(\P\big(\theta_u = x | \mathcal{T}^{(\eps)}_{\infty}\big)
-
\P\big(\theta_u = x | Y^{(\eps)}_{G}\big)\Big)^2\Big] \le 4\sqrt{I\Big(\theta_u;\theta_{\partial S} \Big| Y^{(\eps)}_{S}\Big)\big/ 2}.
\end{align*}
Now we obtain the desired result by averaging over $u \in S$ and $\eps$ and using Jensen's inequality, and then invoking Lemma~\ref{mutual_info}:
\begin{align*}
 \frac{1}{|S|}\sum_{u \in  S} \int_0^{\eps} \sum_{x \in \cX} &\E\Big[\Big(\P\big(\theta_u = x | \mathcal{T}^{(\eps')}_{\infty}\big)-\P\big(\theta_u = x | Y^{(\eps')}_{G}\big)\Big)^2\Big] \rmd \eps' \\
& \le 4\sqrt{\eps \int_0^\eps  \frac{1}{|S|}\sum_{u \in  S} I\Big(\theta_u;\theta_{\partial S} \Big| Y^{(\eps')}_{S}\Big) \rmd \eps\big/ 2}\\
&\le \sqrt{8\eps \log |\cX|  \cdot \frac{|\partial S|}{|S|}}.
\end{align*}
\end{proof}

Now we are in a position to prove Proposition~\ref{triviality}.
\begin{proof}[Proof of Proposition~\ref{triviality}]
Assume $(G,o)$ is almost surely anchored--amenable and tame. Let $(S_{k} = S_k(G,o) )_{k\ge 1}$ be the sequence of finite measurable random subsets of $V(G)$ satisfying the conditions of Definition~\ref{def:tame} (recall in particular that $o\in S_k$.)
  We use Lemma~\ref{lem:point_to_set} with this choice of sequence  $(S_k)_{k\ge 1}$, and then average over the realization of the rooted graph $(G,o)\sim \rho$:
\begin{align}\label{ineq:third}
&\E_\rho\left[\frac{1}{|S_k(G,o)|}\sum_{u \in  S_k(G,o)} \int_{0}^{\eps} \sum_{x \in \cX}  \E\Big[\Big(\P\big(\theta_u = x | \mathcal{T}^{(\eps')}_{\infty}\big) - \P\big(\theta_u = x | Y^{(\eps')}_{G}\big)\Big)^2\Big]  \rmd \eps'\right] \\
&\hspace{3cm}\le  \E_\rho\Big[\sqrt{8\eps \log |\cX|  \cdot |\partial S_k(G,o)| / |S_k(G,o)|}\Big]\nonumber\\
&\hspace{3cm}\le \sqrt{8\eps \log |\cX|}\, \Delta_k,\nonumber
\end{align}
where $\Delta_k\to 0$ by an application of dominated convergence (since $|\partial S_k(G,o)| / |S_k(G,o)|\to 0$ almost surely by assumption).
Now we let $f_k : \mathcal{G}_{**}\mapsto \R_+$ defined by
\[f_k(G,o,u) := \frac{1}{|S_k(G,o)|} \bfone_{u\in S_k(G,o)} \int_{0}^{\eps} \sum_{x \in \cX}  \E\Big[\Big(\P\big(\theta_u = x | \mathcal{T}^{(\eps')}_{\infty}\big) - \P\big(\theta_u = x | Y^{(\eps')}_{G}\big)\Big)^2\Big]  \rmd \eps'.\]
With this notation, expression~\eqref{ineq:third} is equal to $\E_{\rho}\big[\sum_{u \in V(G)}f_k(G,o,u)\big]$.
By unimodularity of $\rho$, this is also equal to  
\begin{align}\label{ineq:unimodularity}
&\E_{\rho}\Big[\sum_{u \in V(G)}f_k(G,u,o)\Big] \nonumber\\
&\qquad \qquad = \E_\rho\left[\Big(\sum_{u \in  V(G)}\frac{\bfone_{o\in S_k(G,u)}}{|S_k(G,u)|}\Big) \int_{0}^{\eps} \sum_{x \in \cX}  \E\Big[\Big(\P\big(\theta_o = x | \mathcal{T}^{(\eps')}_{\infty}\big) - \P\big(\theta_o = x | Y^{(\eps')}_{G}\big)\Big)^2\Big]  \rmd \eps' \right] \nonumber\\
&\qquad \qquad= \E_\rho\left[\alpha_k(G,o) \cdot \int_{0}^{\eps} \sum_{x \in \cX}  \E\Big[\Big(\P\big(\theta_o = x | \mathcal{T}^{(\eps')}_{\infty}\big) - \P\big(\theta_o = x | Y^{(\eps')}_{G}\big)\Big)^2\Big]  \rmd \eps'\right],
\end{align}
where $\alpha_k(G,o) = \sum_{u \in  V(G)}\bfone_{o \in S_k(G,u)} |S_k(G,u)|^{-1}$. Now we use the tameness assumption: the sequence $(1/\alpha_k(G,o))_k$ is tight. Let $Z(G,o)$ be the above integral over $\eps'$ (so that the above display is $\E_{\rho}[\alpha_k(G,o)\cdot Z(G,o)]$.) For $\eta>0$ let $\delta>0$ such that 
\[\limsup_{k \to \infty} \rho(\alpha_k(G,o)\le \delta) \le \eta.\]
Since all involved quantities are nonnegative, we have
\begin{align*}
\E_{\rho}\big[\alpha_k(G,o)\cdot Z(G,o)\big] &\ge \delta \E_{\rho}\big[\bfone_{\alpha_k(G,o)> \delta} \cdot Z(G,o)\big]\\
&= \delta \big( \E_{\rho}\big[Z(G,o)\big] - \E_{\rho}\big[\bfone_{\alpha_k(G,o)\le \delta} \cdot Z(G,o)\big]\big).
\end{align*}
Since $Z(G,o)\le \eps |\cX|$ a.s., we obtain
\begin{align*}
\E_{\rho}\big[Z(G,o)\big] &\le \delta^{-1} \E_{\rho}\big[\alpha_k(G,o) Z(G,o)\big]+ \eps |\cX| \cdot\rho(\alpha_k(G,o)\le \delta)  \\
&\le \delta^{-1} \sqrt{8\eps \log |\cX| }\, \Delta_k + \eps |\cX| \cdot\rho(\alpha_k(G,o)\le \delta),
\end{align*}
where we have used~\eqref{ineq:third} and~\eqref{ineq:unimodularity} to obtain the last display. Letting $k\to \infty$ and then $\eta \to 0$ we obtain for all $\eps \ge 0$,
\[\E_\rho\left[\int_{0}^{\eps} \sum_{x \in \cX}  \E\Big[\Big(\P\big(\theta_o = x | \mathcal{T}^{(\eps')}_{\infty}\big) - \P\big(\theta_o = x | Y^{(\eps')}_{G}\big)\Big)^2\Big]  \rmd \eps'\right] = 0,\]
and this concludes the proof.
\end{proof}

\section*{Acknowledgements}
This work was partially supported by grants NSF DMS-1613091, CCF-1714305, IIS-1741162, and ONR N00014-18-1-2729.

\bibliographystyle{amsalpha}
\bibliography{all-bibliography}

\newcommand{\etalchar}[1]{$^{#1}$}
\providecommand{\bysame}{\leavevmode\hbox to3em{\hrulefill}\thinspace}
\providecommand{\MR}{\relax\ifhmode\unskip\space\fi MR }
\providecommand{\MRhref}[2]{%
  \href{http://www.ams.org/mathscinet-getitem?mr=#1}{#2}
}
\providecommand{\href}[2]{#2}
\begin{thebibliography}{AMM{\etalchar{+}}17}

\bibitem[AB18]{abbe2018information}
Emmanuel Abbe and Enric Boix, \emph{An information-percolation bound for spin
  synchronization on general graphs}, arXiv:1806.03227 (2018).

\bibitem[Abb17]{abbe2017community}
Emmanuel Abbe, \emph{Community detection and stochastic block models: recent
  developments}, The Journal of Machine Learning Research \textbf{18} (2017),
  no.~1, 6446--6531.

\bibitem[ABRS18]{abbe2018graph}
Emmanuel Abbe, Enric Boix, Peter Ralli, and Colin Sandon, \emph{Graph powering
  and spectral robustness}, arXiv:1809.04818 (2018).

\bibitem[AKV02]{alon2002concentration}
Noga Alon, Michael Krivelevich, and Van~H Vu, \emph{On the concentration of
  eigenvalues of random symmetric matrices}, Israel Journal of Mathematics
  \textbf{131} (2002), no.~1, 259--267.

\bibitem[AL07]{aldous2007processes}
David Aldous and Russell Lyons, \emph{Processes on unimodular random networks},
  Electron. J. Probab \textbf{12} (2007), no.~54, 1454--1508.

\bibitem[AMM{\etalchar{+}}17]{abbe2017group}
Emmanuel Abbe, Laurent Massoulie, Andrea Montanari, Allan Sly, and Nikhil
  Srivastava, \emph{Group synchronization on grids}, arXiv:1706.08561 (2017).

\bibitem[AW09]{amini2008high}
Arash~A. Amini and Martin~J. Wainwright, \emph{High-dimensional analysis of
  semidefinite relaxations for sparse principal components}, Annals of
  Statistics \textbf{37} (2009), no.~5B, 2877--2921.

\bibitem[BHK{\etalchar{+}}16]{barak2016nearly}
Boaz Barak, Samuel~B Hopkins, Jonathan Kelner, Pravesh Kothari, Ankur Moitra,
  and Aaron Potechin, \emph{A nearly tight sum-of-squares lower bound for the
  planted clique problem}, 2016 IEEE 57th Annual Symposium on Foundations of
  Computer Science (FOCS), IEEE, 2016, pp.~428--437.

\bibitem[BKM{\etalchar{+}}19]{barbier2019optimal}
Jean Barbier, Florent Krzakala, Nicolas Macris, L{\'e}o Miolane, and Lenka
  Zdeborov{\'a}, \emph{Optimal errors and phase transitions in high-dimensional
  generalized linear models}, Proceedings of the National Academy of Sciences
  \textbf{116} (2019), no.~12, 5451--5460.

\bibitem[Bol80]{bollobas1980probabilistic}
B{\'e}la Bollob{\'a}s, \emph{A probabilistic proof of an asymptotic formula for
  the number of labelled regular graphs}, European Journal of Combinatorics
  \textbf{1} (1980), no.~4, 311--316.

\bibitem[BR13]{berthet2013complexity}
Quentin Berthet and Philippe Rigollet, \emph{Complexity theoretic lower bounds
  for sparse principal component detection}, Conference on Learning Theory,
  2013, pp.~1046--1066.

\bibitem[BS01]{benjamini2001recurrence}
Itai Benjamini and Oded Schramm, \emph{Recurrence of distributional limits of
  finite planar graphs}, Electronic Journal of Probability \textbf{6} (2001),
  1--13.

\bibitem[CM19]{celentano2019fundamental}
Michael Celentano and Andrea Montanari, \emph{Fundamental barriers to
  high-dimensional regression with convex penalties}, {\sf arXiv:1903.10603}
  (2019).

\bibitem[DKMZ11]{decelle2011asymptotic}
Aurelien Decelle, Florent Krzakala, Cristopher Moore, and Lenka Zdeborov{\'a},
  \emph{Asymptotic analysis of the stochastic block model for modular networks
  and its algorithmic applications}, Physical Review E \textbf{84} (2011),
  no.~6, 066106.

\bibitem[DM13]{deshpande2013finding}
Yash Deshpande and Andrea Montanari, \emph{Finding hidden cliques of size
  $\sqrt{N/e}$ in nearly linear time}, Foundations of Computational Mathematics
  (2013), 1--60.

\bibitem[DMS13]{dembo2013factor}
Amir Dembo, Andrea Montanari, and Nike Sun, \emph{Factor models on locally
  tree-like graphs}, The Annals of Probability \textbf{41} (2013), no.~6,
  4162--4213.

\bibitem[EKPS00]{evans2000broadcasting}
William Evans, Claire Kenyon, Yuval Peres, and Leonard~J Schulman,
  \emph{Broadcasting on trees and the ising model}, The Annals of Applied
  Probability \textbf{10} (2000), no.~2, 410--433.

\bibitem[FK00]{feige2000finding}
Uriel Feige and Robert Krauthgamer, \emph{Finding and certifying a large hidden
  clique in a semirandom graph}, Random Structures \& Algorithms \textbf{16}
  (2000), no.~2, 195--208.

\bibitem[FM17]{fan2017well}
Zhou Fan and Andrea Montanari, \emph{How well do local algorithms solve
  semidefinite programs?}, Proceedings of the 49th Annual ACM SIGACT Symposium
  on Theory of Computing, ACM, 2017, pp.~604--614.

\bibitem[HKP{\etalchar{+}}17]{hopkins2017power}
Samuel~B Hopkins, Pravesh~K Kothari, Aaron Potechin, Prasad Raghavendra, Tselil
  Schramm, and David Steurer, \emph{The power of sum-of-squares for detecting
  hidden structures}, 2017 IEEE 58th Annual Symposium on Foundations of
  Computer Science (FOCS), IEEE, 2017, pp.~720--731.

\bibitem[HSS15]{hopkins2015tensor}
Samuel~B Hopkins, Jonathan Shi, and David Steurer, \emph{Tensor principal
  component analysis via sum-of-square proofs}, Conference on Learning Theory,
  2015, pp.~956--1006.

\bibitem[HSSS16]{hopkins2016fast}
Samuel~B Hopkins, Tselil Schramm, Jonathan Shi, and David Steurer, \emph{Fast
  spectral algorithms from sum-of-squares proofs: tensor decomposition and
  planted sparse vectors}, Proceedings of the forty-eighth annual ACM symposium
  on Theory of Computing, ACM, 2016, pp.~178--191.

\bibitem[Jer92]{jerrum1992large}
Mark Jerrum, \emph{Large cliques elude the metropolis process}, Random
  Structures \& Algorithms \textbf{3} (1992), no.~4, 347--359.

\bibitem[JL09]{johnstone2009consistency}
Iain~M Johnstone and Arthur~Yu Lu, \emph{On consistency and sparsity for
  principal components analysis in high dimensions}, Journal of the American
  Statistical Association \textbf{104} (2009), no.~486.

\bibitem[JM04]{janson2004robust}
Svante Janson and Elchanan Mossel, \emph{Robust reconstruction on trees is
  determined by the second eigenvalue}, The Annals of Probability \textbf{32}
  (2004), no.~3B, 2630--2649.

\bibitem[Joh06]{JohnstoneICM}
Iain. Johnstone, \emph{{High Dimensional Statistical Inference and Random
  Matrices}}, Proc. International Congress of Mathematicians (Madrid), 2006.

\bibitem[LM17]{lelarge2017fundamental}
Marc Lelarge and L{\'e}o Miolane, \emph{Fundamental limits of symmetric
  low-rank matrix estimation}, Probability Theory and Related Fields (2017),
  1--71.

\bibitem[LP17]{lyons2017probability}
Russell Lyons and Yuval Peres, \emph{Probability on trees and networks},
  vol.~42, Cambridge University Press, 2017.

\bibitem[Mas14]{massoulie2014community}
Laurent Massouli{\'e}, \emph{Community detection thresholds and the weak
  ramanujan property}, Proceedings of the forty-sixth annual ACM symposium on
  Theory of computing, ACM, 2014, pp.~694--703.

\bibitem[MM06]{mezard2006reconstruction}
Marc M{\'e}zard and Andrea Montanari, \emph{Reconstruction on trees and spin
  glass transition}, Journal of statistical physics \textbf{124} (2006), no.~6,
  1317--1350.

\bibitem[MM09]{MezardMontanari}
Marc M{\'e}zard and Andrea Montanari, \emph{{Information, Physics and
  Computation}}, Oxford, 2009.

\bibitem[MNS18]{mossel2018proof}
Elchanan Mossel, Joe Neeman, and Allan Sly, \emph{A proof of the block model
  threshold conjecture}, Combinatorica \textbf{38} (2018), no.~3, 665--708.

\bibitem[Mon08]{MontanariSparse}
Andrea Montanari, \emph{{Estimating random variables from random sparse
  observations}}, Eur. Trans. on Telecom. \textbf{19} (2008), 385--403.

\bibitem[MP03]{mossel2003information}
Elchanan Mossel and Yuval Peres, \emph{Information flow on trees}, The Annals
  of Applied Probability \textbf{13} (2003), no.~3, 817--844.

\bibitem[MR14]{richard2014statistical}
Andrea Montanari and Emile Richard, \emph{A statistical model for tensor pca},
  Advances in Neural Information Processing Systems, 2014, pp.~2897--2905.

\bibitem[MW15]{ma2015sum}
Tengyu Ma and Avi Wigderson, \emph{Sum-of-squares lower bounds for sparse pca},
  Advances in Neural Information Processing Systems, 2015, pp.~1612--1620.

\bibitem[NS81]{newman1981number}
CM~Newman and LS~Schulman, \emph{Number and density of percolating clusters},
  Journal of Physics A: Mathematical and General \textbf{14} (1981), no.~7,
  1735.

\bibitem[P{\etalchar{+}}03]{penrose2003random}
Mathew Penrose et~al., \emph{Random geometric graphs}, vol.~5, Oxford
  university press, 2003.

\bibitem[PS99]{pemantle1999robust}
Robin Pemantle and Jeffrey~E Steif, \emph{Robust phase transitions for
  heisenberg and other models on general trees}, The Annals of Probability
  (1999), 876--912.

\bibitem[PW18]{polyanskiy2018application}
Yury Polyanskiy and Yihong Wu, \emph{Application of information-percolation
  method to reconstruction problems on graphs}, arXiv:1806.04195 (2018).

\bibitem[SB18]{sankararaman2018community}
Abishek Sankararaman and Fran{\c{c}}ois Baccelli, \emph{Community detection on
  euclidean random graphs}, Proceedings of the Twenty-Ninth Annual ACM-SIAM
  Symposium on Discrete Algorithms, SIAM, 2018, pp.~2181--2200.

\bibitem[Sly11]{sly2011}
Allan Sly, \emph{Reconstruction for the potts model}, The Annals of Probability
  \textbf{39} (2011), no.~4, 1365--1406.

\bibitem[Wor99]{wormald1999models}
Nicholas~C Wormald, \emph{Models of random regular graphs}, London Mathematical
  Society Lecture Note Series (1999), 239--298.

\end{thebibliography}
\addcontentsline{toc}{section}{References}

\appendix

\section{Amenable graphs: some omitted proofs}
\label{sec:omitted}
\subsection{Proof of Proposition~\ref{decoupling_bayes}}
The proof is based on a decoupling principle under $\eps$-perturbation of a general observation channel. This principle is given in Lemma 3.1 in~\cite{MontanariSparse}, which once specialized to our setting, takes the following form:
\begin{lemma}[Lemma 3.1~\cite{MontanariSparse}]\label{Lem:decoupling}
For all $\eps>0$, it holds that
\[\frac{1}{n}  \int_{0}^\eps \sum_{u,v \in V_n} I\Big(\theta_u;\theta_v \big| Y^{(\eps')}_{G_n}\Big) \rmd \eps' \le 2 \log |\cX|.\]
\end{lemma}
This is very similar to our Lemma~\ref{mutual_info}. In fact the latter follows the same line of proof.  

Recall the definition of the decoupled estimator
\begin{align*}
 \widehat{X}^{(\sdec)}_{uv} &:= \E\Big[f(\hat{\theta}_u)\big| Y^{(\eps)}_{G_n}\Big]\cdot\E\Big[f(\hat{\theta}_v)\big|Y^{(\eps)}_{G_n}\Big]\\
 &= \Big(\sum_{x\in \cX} \mu_{u,G_n}(x)f(x)\Big) \cdot \Big(\sum_{x\in \cX} \mu_{v,G_n}(x)f(x)\Big),~~~ u, v \in V_n.
 \end{align*}
 For a pair of vertices $u,v \in V_n$ we let $\mu_{u,v,G_n}(x,x') := \P\big(\theta_u=x,\theta_v=x' \big|Y^{(\eps)}_{G_n}\big)$, for $x,x' \in \cX$.
Expanding the squares and cancelling equal terms we have
\[\risk_n(\widehat{\bX}^{(\sdec)};f) - \risk_n^{\textup{Bayes}}(f) = \frac{1}{n^2} \Big(\E \big\|\bX^{\textup{Bayes}}\big\|_F^2 - \E \big\|\hatbX^{(\sdec)}\big\|_F^2  \Big).\] 
Moreover, 
 \[\E \big\|\hatbX^{(\sdec)}\big\|_F^2 = \sum_{u,v \in V_n}  \E \left[\E\Big[f(\theta_u)\big|Y^{(\eps)}_{G_n}\Big]^2\cdot\E\Big[f(\theta_v)\big|Y^{(\eps)}_{G_n}\Big]^2\right],\]
 and
 \[\E \big\|\widehat{\bX}^{\textup{Bayes}}\big\|_F^2 = \sum_{u,v\in V_n} \E\left[\E\Big[f(\theta_u)f(\theta_v)\big|Y^{(\eps)}_{G_n}\Big]^2\right].\]
 Therefore,
\begin{align*}
&\frac{1}{n^2}\Big(\E \big\|\widehat{\bX}^{\textup{Bayes}}\big\|_F^2  - \E \big\|\hatbX^{(\sdec)}\big\|_F^2 \Big)\\
&\le  \frac{2 \|f\|_{\infty}^2}{n^2} \sum_{u, v \in V_n} \E\left[\left|\E\Big[f(\theta_u)f(\theta_v)\big|Y^{(\eps)}_{G_n}\Big]- \E\Big[f(\theta_u)\big|Y^{(\eps)}_{G_n}\Big]\E\Big[f(\theta_v)\big|Y^{(\eps)}_{G_n}\Big]\right|\right]\\
& \le  \frac{2 \|f\|_{\infty}^2}{n^2} \sum_{u,v \in V_n}  \sum_{x,x' \in \cX} |f(x)f(y)| \E \Big[\Big| \mu_{u,v,G_n}(x,x') -\mu_{u,G_n}(x)\mu_{v,G_n}(x')\Big|\Big]\\
&\le \frac{4\|f\|_{\infty}^4 }{n^2}  \sum_{u,v \in V_n} \E \big[d_{\sTV}( \mu_{u,v,G_n},\mu_{u,G_n}\times \mu_{v,G_n})\big]\\
&\le 4\|f\|_{\infty}^4 \sqrt{\frac{1}{2n^2}  \sum_{u,v \in V_n} I\Big(\theta_u;\theta_v \big| Y^{(\eps)}_{G_n}\Big)}.
\end{align*}
We used Pinsker's inequality and Jensen's inequality in the last line. We apply Lemma~\ref{Lem:decoupling} and Jensen's inequality and obtain for all $\eps>0$,
\[\int_0^\eps \Big\{\risk_n(\widehat{\bX}^{(\sdec)};f) - \risk_n^{\textup{Bayes}}(f)\Big\} \rmd \eps' \le 4\|f\|_{\infty}^4 \sqrt{\frac{\eps \log |\cX|}{n}}\longrightarrow 0.\]
Since the integrand is non-negative, it too converges to zero almost everywhere.

\subsection{Proof of Corollary~\ref{overlap_amenable}}
The proof follows from statement~\eqref{local} of Proposition~\ref{L2_statement} since
\begin{align*}
\E[\overlap(\hbtheta^{(l)},\btheta)] &\ge \frac{1}{|V_n|} \sum_{ u \in V_n}  \P(\hat{\theta}_u = \theta_u) \\
& = \frac{1}{|V_n|} \sum_{ u \in V_n} \sum_{x \in \cX} \E \big[ \bfone_ {\hat{\theta}_i=x}\bfone_ {\theta_i=x}\big]\\
& \stackrel{(a)}{=} \frac{1}{|V_n|} \sum_{ u \in V_n} \sum_{x \in \cX} \E \big[ \widehat{\mu}_{G_n,u,l}(x)\bfone_ {\theta_i=x}\big]\\
& = \frac{1}{|V_n|} \sum_{ u \in V_n} \sum_{x \in \cX} \E \big[ \widehat{\mu}_{G_n,u,l}(x)\mu_{G_n,u}(x)\big]\\
&\stackrel{(b)}{=} \frac{1}{|V_n|} \sum_{ u \in V_n} \sum_{x \in \cX} \E \big[\widehat{\mu}^2_{u,l,G_n}(x)\big].
\end{align*}
Here in $(a)$ we used the fact that, by construction $\htheta_u \sim \widehat{\mu}_{G_n,u,l}(\,\cdot\,)$ and in $(b)$ the remark, already made in the proof
of Theorem \ref{main}, that 
$\E \big[ \widehat{\mu}_{G_n,u,l}(x)\mu_{G_n,u}(x)\big] = \E \big[ \widehat{\mu}_{G_n,u,l}(x)^2\big]$.

\section{Information-theoretic reconstruction on random graphs: Technical proofs}
\label{sec:IT-Zq}

\subsection{Proof of Lemma \ref{lemma:TypicalNonEmpty}}
This is a consequence of McDiarmid's bounded differences inequality. For $(u,v)\in E$ and $(x_1,x_2,y_{12})\in \cX \times \cX \times \cY$, we let $X_{uv}(x_1,x_2,y_{12}) = \indi\{\theta_{0,u} = x_1, \theta_{0,v} = x_2, Y_{uv} = y_{12}\}$, and let $Z(x_1,x_2,y_{12})  = \frac{1}{|E|}\sum_{(u,v) \in E} X_{uv}(x_1,x_2,y_{12})  - \E[X_{uv}(x_1,x_2,y_{12}) ])$.
 Since 
\[d_{\sTV}(\hnu^{G_n}_{\btheta_0,\bY},\onu_{\se}) = \half \sum_{x_1,x_2, y_{12}} \big|Z(x_1,x_2,y_{12})\big|,\] 
we have
\begin{align*}
\P\big(d_{\sTV}(\hnu^{G_n}_{\btheta_0,\bY},\onu_{\se}) \ge \eta\big) &\le \sum_{x_1,x_2,y_{12}} \P\Big( \big|Z(x_1,x_2,y_{12})\big| \ge \frac{2\eta}{|\cX|^2|\cY|}\Big).
\end{align*}
We associate to each edge $(i,j)\in E$ an independent random variable $U_{ij}\sim\Unif([0,1])$. We can then construct a function $f:\cX\times\cX\times [0,1]\to \cY$,
such that $Q(y_{12}|\theta_1,\theta_2) =\prob(f(\theta_1,\theta_2,U_{12})= y_{12}|\theta_1,\theta_2)$. Hence we can define $\bY,\btheta_0$ by letting
$Y_{uv} = f(\theta_{0,u},\theta_{0,v},U_{uv})$ for each $(u,v)\in E$, and we view $Z(x_1,x_2,y_{12})$ as a function of the independent random variables $\{\theta_{0,u},U_{uv}\}$.

Moreover, if we change the value $\theta_{0,u}$ at vertex $u$ to $\theta'_{0,u}$ and call $Z'(x_1,x_2,y_{12})$ the resulting value of $Z(x_1,x_2,y_{12})$, we have $|Z-Z'| \le \frac{k}{|E|} = \frac{2}{n}$ (recall that $k$ is the degree of $u$ and $|E| = \frac{nk}{2}$). If we further change $U_{uv}$ to $U'_{uv}$ at an edge $(uv)\in E$,  we have $|Z-Z'| \le 1/|E|$.
The bounded differences inequality then implies
\[\P\Big( \big|Z(x_1,x_2,y_{12})\big| \ge\eta' \Big) \le 2 \exp\Big\{ - \frac{2\eta'^2}{ (n (\frac{2}{n})^2+|E| (\frac{1}{|E|})^2 ) } \Big\} 
\le 2 e^{-n\frac{\eta'^2}{3}}.\]  
Now we let $\eta' = 2\frac{\eta}{|\cX|^2|\cY|}$ and $\eta = \frac{(\log n)}{\sqrt{n}}$.

\subsection{Proof of Theorem~\ref{exhaustive_general}: A truncated first moment method}
\label{sec:ProofExhaustiveGeneral}

Instead of working directly with the ensemble of random regular graphs, we will use the \emph{configuration model}~\cite{bollobas1980probabilistic} for our moment computations. Let $kn$ be even and let $\mathcal{M}_{nk}$ be the set of perfect matchings on $nk$ vertices. For $\mathfrak{m} \in \mathcal{M}_{nk}$ we define the multi-graph $G(\mathfrak{m})$ on $n$ vertices where a vertex $i' \in [nk]$ in $\mathfrak{m}$ is sent to a vertex $i$ in $G(\mathfrak{m})$ through the mapping $i' \mapsto i = i' ~(mod)~ n$. The resulting multi-graph may contain multiple edges and self-loops. The configuration model is the probability measure $\P^{\scm}_{n,k}$ on multi-graphs induced by the uniform measure on perfect matchings through the above mapping.      
The measure $\P^{\scm}_{n,k}$ conditioned on the multi-graph $G(\mathfrak{m})$ being \emph{simple} (i.e., not having self-loops nor multiple edges) is the uniform measure on $k$-regular graphs $\P^{\sreg}_{n,k}$. The probability that $G(\mathfrak{m})$ is simple under $\P^{\scm}_{n,k}$ is $(1-\bigo(k^3/n))e^{(1-k^2)/4}$ for large $n$ by a formula of McKay and Wormald~\cite{wormald1999models}. Therefore, for any event $A$ and sequence $\eps_n \to 0$, $\P^{\scm}_{n,k}(A)\ge 1-\eps_n$ implies $\P^{\sreg}_{n,k}(A)\ge 1 - c(k)\eps_n$ with $c(k)>0$ depending only on $k$. 


Let $G_n = (V_n,E_n)$ be from the configuration model with $V_n=[n]$.  
We will assume edges to be  directed, and the direction to be chosen uniformly at random. The number of such graphs is
\begin{align}
N_{n,k} = \frac{(nk)!}{(nk/2)!} = \exp\left\{\frac{nk}{2}\log \Big(\frac{2nk}{e}\Big) + \bigo(1)\right\}\, .\label{eq:GraphNumber}
\end{align}
Indeed, $N_{n,k}$ is the number of ordered pairings of the $nk$ half-edges. Such a pairing can be constructed by ordering the $nk$ half-edges
(which can be done in $(nk)!$ possible ways), and then pairing consecutive half edges following this ordering.  Each pairing can arise in $(nk/2)!$ possible ways.

 We next state a standard counting lemma that will be useful in what follows.
Given finite alphabets $\ocX,\ocY$, and integers $n,k$ with $nk = 2m$ even, let $\cuP_k(\ocX\times\ocX\times\ocY)
\subseteq \cuP(\ocX\times\ocX\times\ocY)$ be the subset of probability distributions $\nu\in \cuP(\ocX\times\ocX\times\ocY)$ 
such that $\nu(x_1,x_2,y)\in\naturals/m$ for all $x_1,x_2\in\ocX$, $y\in\ocY$, and $\sum_{\tx,y}(\nu(x,\tx,y)+\nu(\tx,x,y)) \in\naturals/n$ for all $x\in\ocX$.

Given $\nu \in \cuP(\ocX\times\ocX\times\ocY)$, we let $\pi_1\nu(x) \equiv \sum_{\tx\in\ocX,y\in \cY} \nu(x,\tx,y)$,  $\pi_2\nu(x) \equiv \sum_{\tx\in\ocX,y\in \cY} \nu(\tx,x,y)$.
We further let $\pi_{12}\nu(x_1,x_2) = \sum_{y\in\ocY}\nu(x_1,x_2,y)$. 

Recall that Shannon entropy of a probability distribution $p$ on the finite set $\cX$ is $H(p) = -\sum_{x\in S}p(x)\log p(x)$, and the joint empirical edge distribution of $(\btheta, \bY)$ on a graph $G$ is
\[\hnu^G_{\btheta,\bY} = \frac{1}{|E|}\sum_{(u,v)\in E}\delta_{(\theta_{u},\theta_{v},Y_{uv})} \in\cuP(\cX\times\cX \times \cY)\, .\]

\begin{lemma}\label{lemma:CountingBasic}
For such $\nu$, let $N_{n,k}(\nu)$ be the number of triples $(G,\btheta,\bY)$ where $G=(V=[n],E)$ is a graph from the configuration model, $\btheta\in \ocX^V$,
$\bY\in\ocY^E$, with edge empirical distribution equal to $\nu$.  Let $\nu_v\equiv (\pi_1\nu+\pi_2\nu)/2$.
Then 
\begin{align*}
N_{n,k}(\nu) &\le \exp\left\{n A(\nu) + \frac{nk}{2}\log  \Big(\frac{2nk}{e}\Big)\right\}\, ,\label{eq:BoundEdgeType}\\
A(\nu)& \equiv \frac{k}{2}\,H(\nu)-(k-1)H(\nu_v)\, .
\end{align*}
\end{lemma}
\begin{proof}
 Recall that $m=nk/2$ is the number of edges in $G$.
Note that $m\pi_1\nu(x)$ is the number of edges $(u,v)$ such that $\theta_u=x$, and $m\pi_2\nu(x)$ is the number of edges $(u,v)$ such that $\theta_v=x$.
Therefore $m(\pi_1\nu(x)+\pi_2\nu(x))/k = n(\pi_1\nu(x)+\pi_2\nu(x))/2$ is the number of vertices $u$ such that $\theta_u=x$.
Further $m \pi_{12}\nu(x_1,x_2)$ is the number of edges $(u,v)$ such that $\theta_u=x_1$ and
$\theta_v=x_2$.

Given a non-negative integer vector $(b(x))_{x\in S}$ with  $b_{\ssum}\equiv\sum_{x\in S}b(x)$, we denote the corresponding multinomial coefficient by
\begin{align*}
\binom{b_{\ssum}}{b(\,\cdot\,)}\equiv \frac{b_{\ssum}!}{\prod_{x\in S} b(x)}\, .
\end{align*}
We then obtain the following exact counting formula (where $\nu_v(x) \equiv (\pi_1\nu(x)+\pi_2\nu(x))/2$ and $\nu_{12} = \pi_{12}\nu$):
\begin{align}
N_{n,k}(\nu) =& \binom{n}{n\nu_v(\,\cdot\,)}\prod_{x\in\ocX}\frac{[nk\nu_v(x)]!}{\prod_{\tx\in\ocX}[nk\nu_{12}(x,\tx)]! \prod_{\tx\in\ocX}[nk\nu_{12}(\tx,x)]!} 
\prod_{x_1,x_2\in \ocX} \binom{nk\nu(x_1,x_2)/2}{nk\nu(x_1,x_2,\,\cdot\,)/2} \, .\nonumber
\end{align}
The first factor account for the number of ways of choosing $\btheta$. The second corresponds to the ways of giving a matching type to half-edfes.
The third factor counts the number of ways of matching half-edges, and the last one the number of ways of assigning labels in $\ocY$ to edges.

Equation (\ref{eq:BoundEdgeType}) follows by using the following elementary bounds (that hold for any $N\in \naturals$ and any $p\in\cuP(S)$):
\begin{align}
N!\le \Big(\frac{N}{e}\Big)^N\, ,\;\;\;\;\;\; \binom{N}{Np(\,\cdot\,)} \le e^{N\, H(p)}\,.
\end{align}
\end{proof}

Now recall the joint empirical distribution of two assignments $\btheta_0,\btheta\in \cX^V$: 
\[ \homega_{\btheta_0,\btheta} = \frac{1}{|V|}\sum_{u\in V}\delta_{\theta_{0,u},\theta_u} \in \cuP(\cX\times\cX).\]
Further, let $\onu_{\se} (x_1,x_2,y) = \onu(x_1)\, \onu(x_2)\, Q(y|x_1,x_2)$, $\onu$ being the uniform distribution on $\cX$, and
\begin{align*}
\Theta(\eta;G,\bY) = \Big\{\btheta\in \cX^V:\;\; d_{\sTV}(\hnu^G_{\btheta,\bY},\onu_{\se})\le \eta\Big\}\, .
\end{align*}

Given a graph $G$, a true assignment $\btheta_0$,  observations $\bY$, and a closed set $\cS\subseteq \cuP(\cX\times\cX)$ we define
\begin{align}
Z(\cS;G,\btheta_0,\bY)=\left|\Big\{\btheta\in \Theta(\eta_n;G,\bY):\; \homega_{\btheta_0,\btheta}\in \cS\Big\}\right|\, ,
\end{align}
where $\eta_n = (\log n) /\sqrt{n}$.
We denote by $\cG_n$ the set of instances, i.e., triples $(G_n,\btheta_0,\bY)$, where $G_n$ is a graph over $n$ vertices,
$\btheta_0\in \cX^{V_n}$ and $\bY\in \cY^{E_n}$.
\begin{lemma}\label{lemma:Moment}
Assume there exists $c_M>0$ such that $c_M^{-1}\le Q(y|x_1,x_2)\le c_M$ for all $x_1,x_2\in\cX$, $y\in\cY$.
Define the map $S:\cuP(\cX^2\times\cX^2\times\cY)\mapsto \reals$ by 
\begin{align}
S(\Omega)& \equiv \frac{k}{2}H(\Omega)-(k-1)H\big((\pi_1\Omega+\pi_2\Omega)/2\big)-\frac{k}{2}H(\onu_{\se}) +(k-1)H(\onu)\, .
\end{align} 
(Here $\pi_1$, $\pi_2$ are defined as in Lemma \ref{lemma:CountingBasic}, with $\ocX=\cX^2$, and $H$ denotes the Shannon entropy.)
Further define $S_*:\cuP(\cX\times\cX)\mapsto \reals$ by
\begin{align}
S_*(\omega) \equiv  & \max \;\;\;\;\;\;\; S(\Omega)\, ,\label{eq:S-star}\\
& ~ \mbox{\rm subj. to }\;\;\; (\pi_1\Omega+\pi_2\Omega)/2 = \omega\, ,\nonumber\\
& \phantom{\mbox{subj. to }}\;\;  \sum_{\tx_1,\tx_2\in \cX} \Omega(x_1,\tx_1,x_2,\tx_2,y) = \onu_{\se}(x_1,x_2,y)\, ,\nonumber\\
& \phantom{\mbox{subj. to }}\;\;  \sum_{x_1,x_2\in \cX} \Omega(x_1,\tx_1,x_2,\tx_2,y) = \onu_{\se}(\tx_1,\tx_2,y)\, .\nonumber
\end{align}
There is a set $\cG_n^*\subseteq \cG_n$ of `good' instances such that the following happens.
For $\cS\subseteq \cuP(\cX\times\cX)$ a closed set, we have
\begin{align}
\prob\big((G_n,\btheta_0,\bY)\in \cG_n^*\big) & \ge 1-c_0^{-1}\exp\left\{-c_0(\log n)^2 \right\}\, ,\label{eq:ProbGoodSet}\\
\E\left[Z(\cS;G_n,\btheta_0,\bY)\,\bfone_{(G_n,\btheta_0,\bY)\in \cG_n^*}\right]& \le \exp\Big\{n \sup_{\omega\in \cS} S_*(\omega) +C\sqrt{n}\log n \Big\}.\nonumber
\end{align}
\end{lemma}
\begin{proof}
Given a tuple $(G,\btheta_0,\btheta,\bY)$, where $G=(V,E)$ is a graph, $\btheta,\btheta_0\in\cX^V$, $\bY\in\cY^E$, we define its joint edge empirical distribution  $\hOmega^{G}_{\btheta_0,\btheta,\bY} \in\cuP(\cX\times\cX\times\cX\times \cX\times \cY)$ as
\begin{equation}\label{hat_omega}
\hOmega^{G}_{\btheta_0,\btheta,\bY} = \frac{1}{|E|} \sum_{(u,v)\in E}\delta_{\theta_{0,u},\theta_u,\theta_{0,v},\theta_v,Y_{uv}}\, .
\end{equation}
In other words $\hOmega^{G}_{\btheta_0,\btheta,\bY}(x_1,\tx_1,x_2,\tx_2,y_{12})$ is the probability that, sampling an edge $(u,v)\in E$
uniformly at random, we have $\theta_{0,u}=x_1$, $\theta_{u}=\tx_1$, $\theta_{0,v}=x_2$, $\theta_{v}=\tx_2$, $Y_{uv}= y_{12}$.
Let $\cuP_{nk}(\cX^4\times\cY)\subseteq \cuP(\cX^4\times\cY)$ be the subset of probability distributions with entries that are integer multiples of $1/|E|= 2/(nk)$.
For $\Omega \in \cuP_{nk}(\cX^4\times\cY)$, we let $N_{n,k}(\Omega)$ denote the number of tuples with  edge empirical distribution equal to $\Omega$:
\begin{align}
N_{n,k}(\Omega) = \Big|\Big\{(G,\btheta_0,\btheta,\bY):\; \;  \hOmega^{G}_{\btheta_0,\btheta,\bY}=\Omega\Big\}\Big|\, .
\end{align}
Notice that setting $\ocX=\cX\times\cX$, we can view $(\btheta_0,\btheta)$ as a vector in $\ocX^V$ and $\Omega$ 
as a probability distribution in $\cuP(\ocX\times\ocX\times\cY)$. 
Applying Eq.~\eqref{eq:GraphNumber} and Lemma~\ref{lemma:CountingBasic}, we get
\begin{align}
\frac{N_{n,k}(\Omega)}{N_{n,k}} & \le C\, e^{nA(\Omega)}. \label{eq:BasicBoundN}\, 
\end{align}

We define
\begin{align*}
\cG_n^*\equiv \Big\{(G_n,\btheta_0,\bY)\in\cG_n:\;\; d_{\sTV}(\hnu^G_{\btheta_0,\bY},\onu_{\se})\le \eta_n\Big\}\, .
\end{align*}
Then Eq.~\eqref{eq:ProbGoodSet} follows immediately from Lemma~\ref{lemma:TypicalNonEmpty}.

We also define $B_{\sTV}(\onu_{\se};\eta_n)\equiv \{\nu\in \cuP(\cX\times\cX\times \cY):\;\; d_{\sTV}(\nu,\onu_{\se})\le \eta_n\}$.
With this notation
\begin{align*}
Z(\cS;G,\btheta_0,\bY) = \sum_{\btheta\in\cX^V}\bfone_{\hnu^G_{\btheta,\bY}\in B_{\sTV}(\onu_{\se};\eta_n)}\, 
\bfone_{\homega_{\btheta_0,\btheta}\in\cS}\, ,
\end{align*}
and therefore, using Eq.~\eqref{eq:BasicBoundN},
\begin{align}
\label{eq:first_moment}
\E&\left[Z(\cS;G_n,\btheta_0,\bY)\,\bfone_{(G_n,\btheta_0,\bY)\in \cG_n^*}\right]  \\
 & =\frac{1}{N_{n,k}|\cX|^n} \sum_G\sum_{\bY\in\cY^E}\sum_{\btheta_0,\btheta\in\cX^V}\prod_{(u,v)\in E}Q(Y_{uv}|\theta_{0u},\theta_{0v})
\bfone_{\hnu^G_{\btheta_0,\bY}\in B_{\sTV}(\onu_{\se};\eta_n)}\bfone_{\hnu^G_{\btheta,\bY}\in B_{\sTV}(\onu_{\se};\eta_n)}\, 
\bfone_{\homega_{\btheta_0,\btheta}\in\cS}.\nonumber
\end{align}
Recall the definition of $\hOmega^{G}_{\btheta_0,\btheta,\bY}$ from Eq.~\eqref{hat_omega}. We observe that this empirical measure has the following marginals: 
\begin{align*}
\frac{1}{2}(\pi_1\hOmega^{G}_{\btheta_0,\btheta,\bY}+\pi_2\hOmega^{G}_{\btheta_0,\btheta,\bY}) &= \homega_{\btheta_0,\btheta},\\ \sum_{\tx_1,\tx_2\in\cX}\hOmega^{G}_{\btheta_0,\btheta,\bY}(x_1,\tx_1,x_2,\tx_2,y) &= \hnu^G_{\btheta_0,\bY}(x_1,x_2,y),\\ 
\mbox{and} 
\sum_{x_1,x_2\in\cX} \hOmega^{G}_{\btheta_0,\btheta,\bY}(x_1,\tx_1,x_2,\tx_2,y) &= \hnu^G_{\btheta,\bY}(\tx_1,\tx_2,y).
\end{align*}
Moreover, if $Q$ does not vanish, we have
\begin{align*}
\prod_{(u,v)\in E}Q(Y_{uv}|\theta_{0u},\theta_{0v}) &= \exp \Big\{\sum_{(u,v)\in E}\log Q(Y_{uv}|\theta_{0u},\theta_{0v})\Big\} \\
&= \exp\Big\{|E| \int \log Q(y|x_1,x_2) \rmd \hnu^G_{\btheta_0,\bY} (x_1,x_2,y)\Big\}\\
&=: F\big(\hOmega^{G}_{\btheta_0,\btheta,\bY}\big).
\end{align*}
Therefore the summand in the formula~\eqref{eq:first_moment} depends only in the empirical edge distribution $\hOmega^{G}_{\btheta_0,\btheta,\bY}$ of the instance $(G,\btheta_0,\btheta,\bY)$. Now let $\cuQ(\eta_n) \subseteq \cuP(\cX^4\times \cY)$ be the set of $\Omega\in \cuP(\cX^4\times \cY)$ satisfying the constraints
\begin{align}
\label{eq:ConstraintGood}
\begin{cases}
&\frac{1}{2}(\pi_1\Omega+\pi_2\Omega) \in \cS\, ,\\
&\Big(\sum_{\tx_1,\tx_2\in\cX}\Omega(x_1,\tx_1,x_2,\tx_2,y)\Big)_{x_1,x_2,y} \in B_{\sTV}(\onu_{\se};\eta_n)\, ,\\
&\Big(\sum_{x_1,x_2\in\cX}\Omega(x_1,\tx_1,x_2,\tx_2,y)\Big)_{\tx_1,\tx_2,y} \in B_{\sTV}(\onu_{\se};\eta_n)\, .
\end{cases}
\end{align}  
We have
\begin{align*}
\E&\left[Z(\cS;G_n,\btheta_0,\bY)\,\bfone_{(G_n,\btheta_0,\bY)\in \cG_n^*}\right]  \\
&\hspace{3cm}= \frac{1}{N_{n,k}|\cX|^n} \sum_{\Omega\in \cuQ(\eta_n)\cap \cuP_{n,k}(\cX^4\times \cY)} F(\Omega) \sum_{(G,\btheta_0,\btheta,\bY)\in \cG_n} \bfone\big\{\hOmega^{G}_{\btheta_0,\btheta,\bY} = \Omega\big\}\\
&\hspace{3cm} = \frac{1}{N_{n,k}|\cX|^n} \sum_{\Omega\in \cuQ(\eta_n)\cap \cuP_{n,k}(\cX^4\times \cY)} F(\Omega) \, N_{n,k}(\Omega)\\
&\hspace{3cm} \le \frac{C}{|\cX|^n} \sum_{\Omega\in \cuQ(\eta_n)\cap \cuP_{n,k}(\cX^4\times \cY)} F(\Omega) \, e^{nA(\Omega)}\, .
\end{align*}
We applied Lemma~\ref{lemma:CountingBasic} in the last line above.
Due to the second constraint in~\eqref{eq:ConstraintGood}, we can upper bound $F(\Omega)$ as follows 
\begin{align*}
F(\Omega) &\le  \exp\Big\{\frac{nk}{2} \int \log Q(y|x_1,x_2) \rmd \onu_{\se} + Cn\eta_n\Big\}\\
&=\exp\Big\{\frac{nk}{2} \big(-H(\onu_{\se}) + 2H(\onu)\big) + Cn\eta_n\Big\}.
\end{align*}
Therefore, letting 
\begin{align*}
S(\Omega) &= A(\Omega) - \frac{k}{2}H(\onu_{\se}) + (k-1)H(\onu)\\
&= \frac{k}{2}H(\Omega) - (k-1)H\big((\pi_1\Omega+\pi_2\Omega)/2\big)- \frac{k}{2}H(\onu_{\se})+ (k-1)H(\onu),
\end{align*}  
we arrive at 
\begin{align*}
\E&\left[Z(\cS;G_n,\btheta_0,\bY)\,\bfone_{(G_n,\btheta_0,\bY)\in \cG_n^*}\right] \le C \sum_{\Omega\in \cuQ(\eta_n)\cap \cuP_{n,k}(\cX^4\times \cY)} 
\exp\big\{nS(\Omega) + Cn\eta_n\big\}\\
& \le C|\cuP_{n,k}(\cX^4\times \cY)|\exp\big\{n\sup_{\omega\in\cS} S_*(\omega) + C \sqrt{n}\log n\big\}\\
& \le Cn^C \exp\big\{n\sup_{\omega\in\cS} S_*(\omega) + C \sqrt{n}\log n\big\}\, ,
\end{align*}
which implies the claim.
\end{proof}

The next result provides a sufficient condition for weak recovery using the estimator $\hbtheta$ satisfying Eq.~\eqref{eq:IT-Estimator}; this is a more general version of Theorem~\ref{exhaustive_general}.
\begin{customthm}{E}\label{thm:GeneralMoment}
Assume there exists $c_M>0$ such that $c_M^{-1}\le Q(y|x_1,x_2)\le c_M$ for all $x_1,x_2\in\cX$, $y\in\cY$.
Assume $S_*(\onu\times\onu)<-\eps<0$.
Then there exists $\delta = \delta(\eps,c_M)>0$ such that, with probability at least $1- c_0^{-1}\exp\{-c_0(\log n)^2\}$,
the following happens
\begin{align}
d _{\sTV} (\homega_{\hbtheta,\btheta_0},\onu\times\onu)\ge \delta\, .
\end{align}
\end{customthm}
\begin{proof}
Recall that $B_{\sTV}(\onu\times\onu;\delta)$ denotes the set of probability distributions $\omega\in\cuP(\cX\times\cX)$
such that $d _{\sTV} (\omega,\onu\times\onu)\le \delta$. We claim that, under the stated assumptions there exists $\delta,c_1>0$ such that,
setting $\cS_{\delta} = B_{\sTV}(\onu\times\onu;\delta)$, and $\cG_*$ as in Lemma \ref{lemma:Moment}, we have
\begin{align}
\E\left[Z(\cS_{\delta};G_n,\btheta_0,\bY)\,\bfone_{(G_n,\btheta_0,\bY)\in \cG_n^*}\right] \le e^{-c_1n} \, .\label{eq:KeyBound}
\end{align}
Hence, applying Lemma \ref{lemma:Moment}, it follows that, with probability at least $1- c_0^{-1}\exp\{-c_0(\log n)^2\}$ (eventually adjusting the constant $c_0$),
$Z(\cS_{\delta};G_n,\btheta_0,\bY)=0$. Hence $\homega_{\hbtheta,\btheta_0}\not\in \cS_{\delta}$ by construction of $\hbtheta$, and therefore the claim follows.

We are left with the task of proving Eq.~(\ref{eq:KeyBound}), which by Lemma \ref{lemma:Moment} and a continuity argument, follows from $S(\onu\times\onu)<-\eps$.
\end{proof}

The condition $S_*(\onu\times\onu)<-\eps$ might be hard to verify in practice because it requires solving the optimization problem~\eqref{eq:S-star}. We provide a simpler sufficient condition, which is the content of Theorem~\ref{exhaustive_general}:
\begin{lemma}
Let $(\theta_1,\theta_2,Y) \sim \onu_e$, with $\onu_e(x_1,x_2,y) = \onu(x_1)\onu(x_2)Q(y | x_1,x_2)$ . We have $S_*(\onu\times\onu) \le -\frac{k}{2}I(\theta_1,\theta_2;Y)+H(\theta_1)$.
\end{lemma}
\begin{proof}
Let $\Omega_*\in\cuP(\cX^2\times\cX^2\times\cY)$ be any distribution achieving the maximum in  \eqref{eq:S-star} for $\omega = \onu\times\onu$,
and let $(X_1,\tX_1,X_2,\tX_2,Y)$ have distribution $\Omega_*$. Note that $(X_1,X_2,Y)\sim \onu_e$, $(\tX_1,\tX_2,Y)\sim \onu_e$,
$(X_1,\tX_1)\sim \omega$, $(X_2,\tX_2)\sim \omega$, $\omega=\onu\times\onu$, whence
\begin{align*}
S_*(\onu\times\onu) &= S(\Omega_*)\\
& = \frac{k}{2}H(X_1,\tX_1,X_2,\tX_2,Y)-(k-1)H(X_1,X_2) -\frac{k}{2}H(X_1,X_2,Y)+(k-1)H(X_1)\\
& = \frac{k}{2}H(X_1,\tX_1,X_2,\tX_2|Y) +\frac{k}{2}H(Y)-2(k-1)H(X_1) -\frac{k}{2}H(X_1,X_2,Y)+(k-1)H(X_1)\\
& \stackrel{(a)}{\le} \frac{k}{2}H(X_1,X_2|Y) +\frac{k}{2}H(\tX_1,\tX_2|Y) +\frac{k}{2}H(Y)-(k-1)H(X_1) -\frac{k}{2}H(X_1,X_2,Y)\\
& = k H(X_1,X_2,Y) -\frac{k}{2}H(Y)-(k-1)H(X_1) -\frac{k}{2}H(X_1,X_2,Y)\\
& = -\frac{k}{2}I(X_1,X_2;Y)+H(X_1)\, .
\end{align*}
Step $(a)$ follows by sub-additivity of entropy. 
\end{proof}
Hence, if $\frac{k}{2}I(\theta_1,\theta_2;Y) \ge H(X\theta_1)+\eps$, then $S_*(\onu\times \onu)<-\eps<0$, and the claim follows by applying Theorem \ref{thm:GeneralMoment}.

\subsection{Proof of Corollary \ref{overlap_corollary}}

Let $\mathscr{B}_{q \times q}$ is the set of all $q \times q$ non-negative doubly stochastic matrices (with $q = |\cX|$). It holds that 
\begin{equation}
\overlap(\hbtheta,\btheta_0) = \max_{\sigma \in \mathscr{S}_{q}} \sum_{x \in \cX} \homega_{\hbtheta,\btheta_0}(x,\sigma(x)) =
\max_{\pi \in \mathscr{B}_{q \times q}} \sum_{x,x' \in \cX} \pi(x,x')\homega_{\hbtheta,\btheta_0}(x,x').
\end{equation}
Indeed, since the right-most expression in the above display is a linear program, the objective value is maximized at the extreme points of the polytope $\mathscr{B}_{q \times q}$, which by Birkhoff's theorem are permutation matrices: $\pi(x,y) = \bfone_{y = \sigma(x)}$ for $\sigma \in \mathscr{S}_{q}$, hence the equality.    

Since $q\homega_{\hbtheta,\btheta_0} \in \mathscr{B}_{q \times q}$ (we abused notation and identified the joint distribution $\homega_{\hbtheta,\btheta_0}$ on $\cX \times \cX$ with a $q \times q$ matrix), we have 
\[\overlap(\hbtheta,\btheta_0)\ge q \sum_{x,x'} \big(\homega_{\hbtheta,\btheta_0}(x,x')\big)^2.\]
Now, on the event $d_{\sTV}(\homega_{\hbtheta,\btheta_0},\onu\times\onu)\ge \delta$, we have 
$\sum_{x,x'} (\homega_{\hbtheta,\btheta_0}(x,x'))^2 \ge \frac{1}{q^2} + \frac{\delta^2}{q}$. Hence $\overlap(\hbtheta,\btheta_0) \ge \frac{1}{q}+ \delta^2$ on the same event. 

Next, we prove the second statement. For two functions $f,g : \cX \mapsto \R$, we let $\homega_{\hbtheta,\btheta_0}(f,g) := \sum_{x_1,x_2} \homega_{\hbtheta,\btheta_0}(x_1,x_2) f(x_1)g(x_2)$.
Theorem~\ref{exhaustive_general} implies 
\[\P\left(\exists f,g : \cX \mapsto \R \mbox{ s.t. } \big|\homega_{\hbtheta,\btheta_0}(f,g)\big|\ge \frac{\delta}{q(q-1)} \right) \ge 1-c_0^{-1}\exp\{-c_0(\log n)^2\}.\] 
Indeed, if $d_{\sTV}(\homega_{\hbtheta,\btheta_0},\onu\times\onu)\ge \delta$ then there exist $x_1,x_2 \in \cX$ such that $|\homega_{\hbtheta,\btheta_0}(x_1,x_2) - \frac{1}{q^2}| \ge \frac{\delta}{q^2}$. Now take $f = (\delta_{x_1} - \frac{1}{q})\frac{q}{q-1}$ and $g = (\delta_{x_2} - \frac{1}{q})\frac{q}{q-1}$. 

On the other hand, letting $\Fspace := \{ f =  (\delta_{x} - \frac{1}{q})\frac{q}{q-1}, x \in \cX\}$, a union bound implies 
\[\P\left(\exists f,g : \cX \mapsto \R \mbox{ s.t. } \big|\homega_{\hbtheta,\btheta_0}(f,g)\big|\ge \frac{\delta}{q(q-1)} \right) \le q^2 \max_{f,g \in \Fspace}\P\left(\big|\homega_{\hbtheta,\btheta_0}(f,g)\big|\ge \frac{\delta}{q(q-1)} \right).\]   
Therefore, there exists a (deterministic) pair $f,g \in \Fspace$ such that $\P\big(|\homega_{\hbtheta,\btheta_0}(f,g)|\ge \frac{\delta}{q(q-1)}\big) \ge \frac{1-o_n(1)}{q^2} >c_0 >0$. By Markov's inequality, this in turn implies that for this specific pair $f,g \in \Fspace$ we have
\begin{equation}\label{lb_overlap}
\E \big[\homega_{\hbtheta,\btheta_0}(f,g)^2\big] \ge c_0 \frac{\delta^2}{(q(q-1))^2} = c(q)\delta^2.
\end{equation}
Now consider estimating the matrix $\bX_f$ (recall that $(X_f)_{uv} = f(\theta_u)f(\theta_v)$) with the matrix $\hatbX^{(\lambda)}$ having entries 
$\widehat{X}^{(\lambda)}_{uv} = \lambda g(\hat{\theta}_u)g(\hat{\theta}_v)$, with $\lambda= n^2 \E \big[\homega_{\hbtheta,\btheta_0}(f,g)^2\big] \big/ \E\big[\|\hatbX^{(1)}\|_F^2\big]$. 
Since 
\[\frac{1}{n^2}\lbr \hatbX^{(\lambda)}, \bX_f \rbr  = \frac{\lambda}{n^2} \sum_{u,v \in V_n} f(\theta_u)f(\theta_v)g(\hat{\theta}_u)g(\hat{\theta}_v) = \lambda\homega_{\hbtheta,\btheta_0}(f,g)^2,\]
the loss $\risk_n$ incurred is
\begin{align*}
\risk_n\big(\hatbX^{(\lambda)};f\big) &= \frac{1}{n^2} \E \|\bX_f\|_F^2 - 2 \lambda \E \big[\homega_{\hbtheta,\btheta_0}(f,g)^2\big] +  \frac{\lambda^2}{n^2} \E \|\hatbX^{(1)}\|_F^2\\
&= \frac{1}{n^2} \E \|\bX_f\|_F^2 - \frac{\E \big[\homega_{\hbtheta,\btheta_0}(f,g)^2\big]^2}{(\E \|\hatbX^{(1)}\|_F^2/n^2)}.
\end{align*}
We have $\E \|\bX_f\|_F^2 = \sum_{u,v\in V_n} \E[f(\theta_u)^2f(\theta_v)^2] = \frac{n}{q}\sum_{x \in \Z_q}f(x)^4+ n(n-1)$. So $\lim \frac{1}{n^2}\E \|\bX_f\|_F^2= 1$. Furthermore, since $\|g\|_{\infty}=1$, $\E \|\hatbX^{(1)}\|_F^2 =\sum_{u,v\in V_n} \E[g(\hat{\theta}_u)^2g(\hat{\theta}_v)^2] \le n^2$. Combining these estimates with the lower bound~\eqref{lb_overlap} implies $\limsup \risk_n \big(\hatbX^{(\lambda)};f\big) < 1 -c(q) \delta^2$. Since $\risk_n^{\textup{Bayes}}(f)\le \risk_n \big(\hatbX^{(\lambda)};f\big)$ this concludes the proof.

\section{Local algorithms on random graphs: Technical proofs}
\label{sec:Loc-Zq}

\subsection{Proof of Theorem~\ref{local_alg_Z_q}}
\label{sec:Proof_local_alg_Z_q}

\subsubsection{Preliminaries}
Let $(T_k,o)$ denote the infinite $k$-regular tree rooted at $o$. (Except the root $o$, every vertex has $k-1$ offsprings.) 
By expanding  the square, we get  
\begin{align*}
\E \big[d_{\sTV}(\widehat{\mu}_{G_n,u,l},\onu)^2\big]\le \frac{q}{4}\E \big[d_{\ell_2}(\widehat{\mu}_{G_n,u,l},\onu)^2\big] = \frac{q}{4}
\left(\sum_{x \in \Z_q} \E\{\widehat{\mu}_{G_n,u,l}(x)^2\} - \frac{1}{q}\right)\,.
\end{align*}
(Here, $d_{\ell_2}$ is the $\ell_2$ distance in $\R^q$.)     
Since the graph sequence $(G_n)_{n \ge 1}$ almost surely converges locally--weakly to (a Dirac delta on) $(T_k,o)$, we have 
\begin{equation}\label{l2_distance}
\lim_{n \to \infty} \frac{1}{|V_n|}\sum_{u\in V_n}\E \big[d_{\ell_2}(\widehat{\mu}_{G_n,u,l},\onu)^2\big] = \sum_{x \in \Z_q} \E\Big[\P\big(\theta_o = x | \b Y^{(\eps)}_{B_{T_k}(o,l)}\big)^2\Big] - \frac{1}{q}.
\end{equation}
Recall 
\[\mu_{o,l}(x) =\P\Big(\theta_o = x | Y^{(\eps)}_{B_{T_k}(o,l)}\Big)~~~\mbox{for all}~ x \in \Z_q.\]
Let $Q_{x} = \text{Law}(\mu_{o,l} | \theta_o =x,\xi^{(\eps)}_o=\star)$ be the conditional law of $\mu_{o,l}$ given the value at the root being $x$ and no information revealed by the side channel. This is a probability distribution on the simplex $\Delta^{q-1}=\cuP(\Z_q)$: $Q_x \in \cuP(\Delta^{q-1})$. 
 Furthermore, let $Q = \text{Law}(\mu_{o,l}|\xi^{(\eps)}_o = \star)= \frac{1}{q}\sum_{x \in \Z_q} Q_x$. The following simple lemma from~\cite{mezard2006reconstruction} is quite useful.
\begin{lemma}\label{density}
For every $x\in \Z_q$, $Q_x$ has a density w.r.t.\ $Q$, and $\frac{\rmd Q_{x}}{\rmd Q} (\mu) = q \mu(x)$ for all $\mu \in \Delta^{q-1}$. 
\end{lemma}
\begin{proof}
Let $\psi :\Delta^{q-1} \to \R$ be bounded measurable. We let $Y \equiv \{Y^{(\eps)}_{B_{T_k}(o,l)}\}$. Then
\begin{align*}
\E\big[\psi(\mu_{o,l})|\theta_o=x,\xi^{(\eps)}_o = \star\big] &= q\E\big[\psi(\mu_{o,l})\bfone\{\theta_o=x\}|\xi^{(\eps)}_o = \star\big] \\
&= q \E\big[\E\big[\psi(\mu_{o,l})\bfone\{\theta_o=x\}|Y,\xi^{(\eps)}_o = \star\big]|\xi^{(\eps)}_o = \star\big]\\
&= q \E\big[\psi(\mu_{o,l})\E\big[\bfone\{\theta_o=x\}|Y,\xi^{(\eps)}_o = \star\big]|\xi^{(\eps)}_o = \star\big]\\
&= q \E\big[\psi(\mu_{o,l})\mu_{o,l}(x)|\xi^{(\eps)}_o = \star\big].
\end{align*}
Therefore $\rmd Q_{x} / \rmd Q (\mu) = q \mu(x)$.
\end{proof}

With the above lemma in hand, the right-hand side in~\eqref{l2_distance} can be written as
\begin{align*}
&\lim_{n \to \infty} \frac{1}{|V_n|}\sum_{u\in V_n}\E \big[d_{\ell_2}(\widehat{\mu}_{G_n,u,l},\onu)^2\big] \\
&\hspace{2cm}=\eps \sum_{x} \E[\bfone\{x = \theta_o\}|\xi_o\neq\star] + (1-\eps)\sum_{x} \E[\mu_{o,l}(x)^2|\xi_o=\star] - \frac{1}{q}\\
&\hspace{2cm}=\eps + \frac{1-\eps}{q}\sum_{x} \E[\mu_{o,l}(x)|\theta_o=x,\xi_o=\star] - \frac{1}{q}\\
&\hspace{2cm}=\eps \frac{q-1}{q} + (1-\eps) \Big(\E[\mu_{o,l}(\theta_o) |\xi_o=\star] - \frac{1}{q}\Big).
\end{align*}
The first equality follows by conditioning on $\xi^{(\eps)}_o$ as noting that conditional on $\xi^{(\eps)}_o \neq \star$, $\mu_{o}(x) = \bfone\{x=\xi_o^{(\eps)}\}$. Lemma~\ref{density} was used to obtain the second equality. 

In light of the above expression, we will track the evolution of the sequence 
\[\hat{z}_{o,l} := \E[\mu_{o,l}(\theta_o) |\xi_o=\star] - \frac{1}{q},~~~ l\ge 0,\] 
which measures the deviation from uniformity of the local marginal at the root. In order to exploit the recursive structure of the tree, we will need to work at the level of the first offsprings of $o$. 
For every offspring $u$ of $o$, we denote by $T^{\downarrow}(u,l)$ the first $l$ generations of the subtree rooted at $u$ not containing $o$; this is a $(k-1)$--ary tree. Now, (with a slight notation override) we redefine 
\[\mu_{u,l}(x) :=\P\Big(\theta_u = x | Y^{(\eps)}_{T^{\downarrow}_k(u,l)}\Big)~~~\mbox{for all}~ x \in \Z_q,\]  
and consider the auxiliary sequence  
\[z_{l} := \E[\mu_{u,l}(\theta_u)|\xi_u=\star]-\frac{1}{q},~~~ l\ge 0.\]
Note that the above definition does not depend on $u$ since $\mu_{u,l}(\theta_u)$ have the same distribution for all $u \sim o$.     
In the next proposition, we relate the two sequences $(\hat{z}_{o,l})_{l\ge 0}$ and $(z_l)_{l\ge 0}$, and establish a recursion for the latter.
\begin{proposition}\label{error_recursion}
Let $\kappa = (k-1)(1-p)^2$ and $\hat{\kappa} = k(1-p)^2$. There exists constants $c, C>0$ depending only on $q$ such that the following holds. If for some $l\ge1$, $\hat{\kappa}|z_{l-1}| \le c$ and $\hat{\kappa}\eps \le c$, then 
\begin{align*}
\Big| \hat{z}_{o,l} - \eps\hat{\kappa} \frac{q-1}{q} - (1-\eps)\hat{\kappa}  z_{l-1}\Big| &\le C \hat{\kappa}^2 (z_{l-1}^2+\eps^2),\\
\mbox{and}~~~~~
\Big| z_{l} - \eps\kappa \frac{q-1}{q} - (1-\eps)\kappa  z_{l-1}\Big| &\le C \kappa^2 (z_{l-1}^2+\eps^2).
\end{align*}
\end{proposition}
The proof of this proposition is presented in Section \ref{sec:error_recursion}
Theorem~\ref{local_alg_Z_q} follows directly from Proposition \ref{error_recursion}, as shown in the next Corollary.
\begin{corollary}
If $\kappa <1$ and $\hat{\kappa} \eps <c$ for a constant $c=c(q,\kappa)$ then there exists $L = L(q,\kappa)$ such that $|\hat{z}_{o,l}| \le L\eps$ for all $l \ge 0$.
\end{corollary}
\begin{proof}
We only need to prove that $|z_l| \le L \eps$, which we will achieve by induction. Since $z_0 = 0$, let's assume that $|z_l| \le L \eps$ for a fixed $l \ge 0$. Then we obtain from Proposition~\ref{error_recursion} that
\[|z_{l+1}| \le \eps\kappa \frac{q-1}{q} + \kappa L \eps +  C \kappa^2 (L^2+1)\eps^2.\]  
It suffices to find an $L$ (independent of $\eps$) such that the above upper bound is smaller than $L \eps$ for all $\eps$. This is equivalent to the quadratic inequality $\kappa \frac{q-1}{q} + C \kappa^2\eps - (1-\kappa) L + C \kappa^2 \eps L^2 \le 0$. The smallest solution to this inequality is $L_*=\frac{1-\kappa - \sqrt{\Delta}}{2\alpha}$, with $\alpha = C\kappa^2\eps$, and $\Delta = (1-\kappa)^2 - 4\alpha(\alpha+ \frac{q-1}{q}\kappa)$. Latter is non-negative provided that $\eps < c_0(q)(\frac{1}{\kappa}-1)^2$ for constant some $c_0(q)>0$. Moreover, for $\eps$ small enough we can write $\sqrt{\Delta} = (1-\kappa)(1-2\alpha (\alpha+\frac{q-1}{q}\kappa)/(1-\alpha)^2) + \bigo(\eps^2)$, so that $L_* = (\alpha + \frac{q-1}{q}\kappa)/(4(1-\kappa)) + \bigo(\eps^2)$. Therefore, we can take $L = (C + \frac{q-1}{q}\kappa)/(4(1-\kappa)) + 1$. 
\end{proof}

\subsubsection{Proof of Proposition \ref{error_recursion}: Analysis of the recursion on the tree}
\label{sec:error_recursion}

Here, we prove Proposition~\ref{error_recursion}. The two statements can be treated in exactly the same way; the only difference being that the root $o$ has $k$ children, while every other vertex has $k-1$ children. For this reason we only write a detailed proof for the first statement; the second one is obtained merely by replacing $k$ by $k-1$.     

Observe that conditional on $\xi_o^{(\eps)} = \star$ the marginal at $o$ is obtained from the marginals at its offsprings $u \sim o$ by a sum-product relation which, in the case of $\Z_q$--synchronization, has the form
\begin{align}
\mu_{o,l}(x) &= \frac{1}{\Sigma_{o,l}} \prod_{u\sim o} \sum_{y \in \Z_q} M_{x,y}(Y_{ou}) \mu_{u,l-1}(y) \nonumber\\
&=\frac{1}{\Sigma_{o,l}} \prod_{u\sim o}  \Big( \frac{p}{q} + (1-p) \mu_{u,l-1}(x-Y_{ou})\Big).\label{tree_recursion}
\end{align}
where $\Sigma_{o,l}$ is the normalizing constant, and $M_{x,y}(Y_{ou})= \P(\theta_o = x | \theta_u=y ,Y_{ou}) =\frac{p}{q} + (1-p)\bfone\{Y_{ou} = x-y\}$ is the Markov transition matrix associated to a `broadcasting process' on the tree according to the $\Z_q$--synchronization model.   

The recursion~\eqref{tree_recursion} induces a deterministic recursion over probability distributions over the simplex $\Delta^{q-1}=\cuP(\Z_q)$. 
Namely, if we define  $Q^{(l)}_{x} := \text{Law}(\mu_{o,l} | \theta_o =x)\in \cuP(\Delta^{q-1})$, we obtain a recursion that determines 
$Q^{(l)}_{x}$ in terms of $Q^{(l-1)}_{x}$ (notice that, by Lemma \ref{density}, once $Q^{(l)}_{x}$ is given for one value of $x$, it is determined for the other values as well.) 
The laws of $\mu_{u,l-1}$ are given by $\text{Law}(\mu_{u,l-1} | \theta_u =x) = Q^{(l-1)}_{x}$ for all $u\sim o$. Note that this law does not depend on $u$ since $\mu_{u,l-1}$ are i.i.d.\ given $\theta_o$. Then $Q^{(l)}_{x}$ can be obtained from $Q^{(l-1)}_{x}$ as follows: 
\begin{enumerate}
\item Draw $\theta_o$ and $\theta_u, \forall u \sim o$ independently and uniformly at random from $\Z_q$.
\item Construct $\{Y_{ou}, u \sim o\}$ according to the $\Z_q$--synchronization model~\eqref{Z_q_model}.
\item Draw $\mu_{u,l-1}$ from $Q^{(l-1)}_{\theta_u}$ independently for each $u\sim o$. 
\item Construct a distribution $\mu$ according to~\eqref{tree_recursion}.
\item Then, given $\xi^{(\eps)}_{o} =\star$, $\mu_{o,l}$ has the same law as $\mu$.
\end{enumerate}

We now analyze the map described above. Define
\[Z_o (x) := \prod_{u \sim o} \big(p + (1-p)q\mu_u(x-Y_{ou})\big),\]
so $\mu_o(x) = Z_o(x)/\sum_{y} Z_o(y)$, where we have dropped the indices $l$ for convenience. Following the analysis of~\cite{sly2011}, we use the identity $\frac{a}{b+c} = \frac{a}{b} - \frac{ac}{b^2} + \frac{c^2}{b^2} \frac{a}{b+c}$ with $a = Z_o(x)$, $b = q$ and $c = \sum_{y} Z_o(y) - q$ to write
\begin{equation}\label{decomposition}
\mu_o(x) = \frac{1}{q}Z_o(x) - \frac{1}{q^2} Z_o(x)\Big(\sum_y Z_o(y) - q\Big) + \Big(\frac{1}{q}\sum_y Z_o(y) -1\Big)^2\frac{Z_o(x)}{\sum_y Z_o(y)}.
\end{equation}
Next we compute the conditional expectations of $Z_o(y)$ and $Z_o(y)Z_o(y')$ (given $\theta_o = x$ and $\xi_o = \star$) in order to control $\E [\mu_o(\theta_o)| \xi_o = \star]$. 
\begin{lemma}\label{computations}
Let $\delta_u := \mu_{u,l} - \frac{1}{q}$ for $u \sim o$. For all $x,y,y' \in \Z_q$, we have
\begin{align}\label{order_one}
\E [Z_o(y)| \theta_o =x,\xi_o=\star] &=\Big(1+ \eps(1-p)^2q \big(\bfone_{y=x}-\frac{1}{q}\big) \\
&~~~~~~~+(1-\eps)(1-p)^2q \E[\delta_u(y-x+\theta_u)|\xi_u=\star]\Big)^{k},\nonumber 
\end{align}
and
\begin{align}\label{order_two} 
\E[Z_o(y)Z_o(y') | \theta_o = x,\xi_o=\star] &= \Big(1 + \eps p(1-p)^2q  \Big(\big(\bfone_{y=x}-\frac{1}{q}\big) +  \big(\bfone_{y'=x}-\frac{1}{q}\big)\Big)\\
&~~~~+ \eps (1-p)^2q \big(\bfone_{y=y'}-\frac{1}{q}\big)\nonumber\\
&~~~~+ \eps (1-p)^3q^2  \bfone_{y=y'}\big(\bfone_{y=x}-\frac{1}{q}\big)\nonumber\\
&~~~~+ (1-\eps)(1-p)^2q\E[\delta_u(y-x+\theta_u)| \xi_u=\star] \nonumber\\
&~~~~+ (1-\eps)(1-p)^2q \E[\delta_u(y'-x+\theta_u)| \xi_u=\star]\nonumber\\
&~~~~+ (1-\eps)p(1-p)^2q\sum_{z}\E[ \delta_u(y-z)\delta_u(y'-z) | \xi_u = \star] \nonumber\\
&~~~~+ (1-\eps)(1-p)^3q^2\E[\delta_u(y-x+\theta_u)\delta_u(y'-x+\theta_u) |\xi_u=\star]\Big)^{k}.\nonumber
\end{align}
\end{lemma}
\begin{proof} We start with the first identity~\eqref{order_one}. 
Since the distributions $\{(\mu_u,Y_{ou}) :\; u \sim o\}$ are conditionally independent given $\theta_o$, we have 
\[\E [Z_o(y)| \theta_o = x,\xi_o=\star] = \prod_{u \sim o} \big(p + (1-p)q\E[\mu_u(y-Y_{ou})|\theta_o = x,\xi_o=\star]\big).\]
Moreover, 
\begin{align*}
\E[\mu_u(y-Y_{ou})| \theta_o = x,\xi_o=\star] &= \eps\E[\bfone_{y - Y_{ou}=\xi_u}|\theta_o=x,\xi_o = \star,\xi_u\neq\star] \\
&~~+(1-\eps) \E[\mu_u(y-Y_{ou})|\theta_o=x,\xi_o = \star,\xi_u = \star].
\end{align*}
The first term in the right-hand side is $\P(y - Y_{ou}=\theta_u|\theta_o=x) = (1-p)\bfone_{x=y} + \frac{p}{q}$. 
The second term is 
\begin{align*}
(1-p)\E[\mu_u(y-x+\theta_u)|&\xi_u = \star] + \frac{p}{q}\sum_{z}\E[\mu_u(y-x+z)|\xi_u = \star]\\
&=(1-p)\E[\mu_u(y-x+\theta_u)|\xi_u = \star] + \frac{p}{q}.
\end{align*}
Therefore 
\begin{align*}
\E[\mu_u(y-Y_{ou})| \theta_o = x,\xi_o=\star] &= \eps\big((1-p)\bfone_{x=y} + \frac{p}{q}\big)\\
&~~~+(1-\eps)\big((1-p)\E[\mu_u(y-x+\theta_u)|\xi_u = \star]+\frac{p}{q}\big)\\
&=\frac{1}{q}+\eps(1-p)(\bfone_{x=y}-\frac{1}{q})\\
&~~~+(1-\eps)(1-p)\E[\delta_u(y-x+\theta_u)|\xi_u = \star].
\end{align*}
So we obtain
\begin{align*}
\E [Z_o(y)| \theta_o =x,\xi_o=\star] &= \prod_{u \sim o}\big(1+ \eps(1-p)^2(q\bfone_{x=y}-1)+ (1-\eps)(1-p)^2q\E[\delta_u(y-x+\theta_u)|\xi_u = \star]\big) \\
&=\Big(1+ \eps(1-p)^2(q\bfone_{x=y} -1)+ (1-\eps)(1-p)^2q\E[\delta_u(y-x+\theta_u)|\xi_u = \star]\Big)^{k},
\end{align*}
where $u$ is an arbitrary offspring since terms participating in the product are all equal. 
Now we deal with the second identity~\eqref{order_two}:
\begin{align*}
\E[Z_o(y) Z_o(y')|\theta_o = x,\xi_o=\star] &= \prod_{u \sim o}\E\Big[\big(p+(1-p)q\mu_u(y-Y_{ou})\big)\\
&\hspace{2cm}\big(p+(1-p)q\mu_u(y'-Y_{ou})\big)|\theta_o=x,\xi_{o}=\star\Big]\\
&= \Big(p^2+ p(1-p)q\big(\E[\mu_u(y-Y_{ou})| \theta_o=x,\xi_{o}=\star] \\
&\hspace{3cm}~~+ \E[\mu_u(y'-Y_{ou})| \theta_o=x,\xi_{o}=\star]\big) \\ 
&~~~~~~+ (1-p)^2q^2\E[\mu_u(y-Y_{ou})\mu_u(y'-Y_{ou}) |\theta_o=x,\xi_{o}=\star]\Big)^{k}.
\end{align*}
Similarly to a previous computation, we have 
\begin{align*}
\E[\mu_u(y - Y_{ou})|\theta_o=x,\xi_o=\star] &= \eps\Big((1-p)\bfone_{x=y} + \frac{p}{q}\Big)\\
&~~~+(1-\eps)\Big((1-p)\E[\mu_u(y-x+\theta_u)|\xi_u=\star]+\frac{p}{q}\Big),
\end{align*}
and
\begin{align*}
\E[\mu_u(y - Y_{ou})\mu_u(y' - Y_{ou}) &| \theta_o=x,\xi_o=\star] 
= \eps \P(y - Y_{ou} = y'-Y_{ou} = \theta_u |\theta_o=x)\\
&+(1-\eps)\E[\mu_u(y - Y_{ou})\mu_u(y' - Y_{ou})|\theta_o=x,\xi_o=\star,\xi_u=\star]\\
&=\eps\bfone_{y=y'}\Big( (1-p)\bfone_{y=x}+\frac{p}{q}\Big)\\
&~~~+(1-\eps)\Big((1-p) \E[\mu_u(y-x+\theta_u)\mu_u(y'-x+\theta_u)|\xi_u=\star]\\
&\hspace{2cm}~ + \frac{p}{q}\sum_{z} \E[\mu_u(y-z)\mu_u(y'-z)|\xi_u=\star]\Big).
\end{align*}
Combining and rearranging terms we obtain the desired result.
\end{proof}

Now we use the expressions just obtained to produce Taylor estimates for each term in the decomposition~\eqref{decomposition}.  
\begin{lemma}\label{taylor_approx}
Let $X =  \E[\delta_u(\theta_u)| \xi_u = \star]$ and $\hat{\kappa} = k(1-p)^2$. There exists constants $c, C$ depending only on $q$ such that if $\hat{\kappa}|X| \le c$ and $\hat{\kappa} \eps <  c$, then  
\begin{align}
\Big|\E[Z_o(x)|\theta_o=x,\xi_o = \star] - 1- \eps\hat{\kappa} (q-1) - (1-\eps)\hat{\kappa} q X\Big| &\le C \hat{\kappa}^2 (X^2+\eps^2),\label{expansion_1}\\
\Big|\E\Big[Z_o(x)\Big(\sum_y Z_o(y) - q\Big)\Big|\theta_o=x,\xi_o = \star\Big]\Big| &\le C \hat{\kappa}^2 (X^2+\eps^2),\label{expansion_2}\\
\mbox{and}~~\E\Big[\Big(\sum_{y\in\Z_q} Z_o(y) -q\Big)^2\Big|\theta_o=x,\xi_o=\star\Big] &\le C \hat{\kappa}^2 (X^2+\eps^2).\label{expansion_3}
\end{align}
\end{lemma}
\begin{proof}
We use $|(1+x)^d - 1 - x| \le e^c d^2 x^2$ for all $x$ such that $d |x|\le c$. Applying this to~\eqref{order_one} yields~\eqref{expansion_1}. For $k(1-p)^2q|X| \le 1/2$ and $k(1-p)^2 (q-1) \eps<1/2$ we have
\begin{align*}
\Big|\E[Z_o(x)|\theta_o=x,\xi_o=\star] - 1-\eps\hat{\kappa} (q-1) - (1-\eps)\hat{\kappa} q X \Big| &\le e \hat{\kappa}^2\big(\eps(q-1)+(1-\eps)qX\big)^2\\
&\le 2e \hat{\kappa}^2q^2\big(\eps^2+X^2\big).
\end{align*}
Next, we use~\eqref{order_two}, combined with the fact $\sum_{y}\delta_u(y) = 0$ to obtain that if $\hat{\kappa}(|X|\vee \eps) \le c(q)$ for some constant $c(q)$ then 
\[\Big|\sum_y \E\big[Z_o(x) Z_o(y)|\theta_o=x,\xi_o=\star] - q - \eps \hat{\kappa} q(q-1) - (1-\eps)\hat{\kappa} q^2X\Big| \le C(q)k^2\Sigma^2,\]
where $\Sigma$ gathers all the terms other than 1 in the expression~\eqref{order_two}, and the constant $C$ depends on $c$. 
We use the inequality $(\sum_{i=1}^n x_i)^2 \le n\sum_i x_i^2$ to obtain
\begin{align}\label{error_2}
k^2\Sigma^2 &\le C(q)\hat{\kappa}^2\Big(\eps^2+ \E[\delta_u(\theta_u)|\xi_u=\star]^2\nonumber\\
&~~~+ \sum_{y \in\Z_q} \E[\delta_u(y-x+\theta_u)|\xi_u=\star]^2 + \max_{z\in\Z_q} \E[\delta_u(z)^2|\xi_u=\star]^2\Big).
\end{align}
The last term was obtained by using Cauchy-Schwarz on the term $\sum_{z}\E[ \delta_u(y-z)\delta_u(y'-z) | \xi_u = \star]$ in~\eqref{order_two} and then replacing sums over $y,z$ by maxima. 
Now it remains to show that the last two terms in~\eqref{error_2} are bounded by $X^2$. Starting with the last term, we have 
\[\max_{z\in\Z_q} \E[\delta_u(z)^2|\xi_u=\star]^2 \le \Big(\sum_{z\in\Z_q} \E[\delta_u(z)^2|\xi_u=\star]\Big)^2 
= \E[\delta_u(\theta_u)|\xi_u=\star]^2 = X^2.
\]
 As for the remaining term,
\begin{lemma}\label{maximal}
We have $\sum_{y \in\Z_q} \E[\delta_u(y-x+\theta_u)|\xi_u=\star]^2 \le qX^2$.
\end{lemma}
This implies $k^2\Sigma^2 \le C(q)\hat{\kappa}^2(\eps^2+X^2)$.
This, combined with~\eqref{expansion_1}, allows us to deduce~\eqref{expansion_2}.
Now we treat the last term~\eqref{expansion_3}:
\begin{align*}
\E\Big[\big(\sum_y Z_o(y) -q\big)^2|\theta_o=x,\xi_o=\star\Big] &= \sum_{y,y'}\E[ Z_o(y) Z_o(y')|\theta_o=x,\xi_o=\star] \\
&~~~- 2q\sum_{y}\E[Z_o(y)|\theta_o=x,\xi_o=\star] + q^2.
\end{align*}
Similarly to our treatment of the quantity $\Sigma$, we use expression~\eqref{order_two} and perform a Taylor expansion to obtain
\[\Big| \sum_{y,y'}\E[ Z_o(y) Z_o(y')|\theta_o=x,\xi_o=\star] - q^2 \Big| \le C(q)\hat{\kappa}^2(\eps^2+X^2).\]
Using~\eqref{order_one} the cross term can be estimated as
\begin{align*}
\Big|\sum_{y \in \Z_q}\E[Z_o(y)|\theta_o=x,\xi_o=\star] - q \Big| \le C(q)\hat{\kappa}^2(\eps^2+X^2).
\end{align*}
Now we conclude 
\[\E\Big[\big(\sum_{y\in\Z_q} Z_o(y) -q\big)^2|\theta_o=x,\xi_o=\star\Big]\le C(q)\hat{\kappa}^2(\eps^2+X^2).\]
\end{proof}

\begin{proofof}{Lemma~\ref{maximal}}
 For $y \in \Z_q$, using Lemma~\ref{density} we have $\E[\delta_u(y+\theta_u)|\xi_u=\star] = \frac{1}{q}\sum_{z\in \Z_q} \E[\delta_u(y+z)|\theta_u = z,\xi_u=\star] = \sum_{z\in \Z_q} \E[\delta_u(y+z)\mu_u(z)|\xi_u=\star] = \sum_{z\in \Z_q} \E[\delta_u(y+z)\delta_u(z)|\xi_u=\star]$. The last equality follows from $\sum_z \delta_u(z) = 0$. Then
 \begin{align*}
 \sum_{y \in\Z_q} \E[\delta_u(y+\theta_u)|\xi_u=\star]^2 &= \sum_{y \in\Z_q} \Big(\sum_{z\in \Z_q} \E[\delta_u(y+z)\delta_u(z)|\xi_u=\star]\Big)^2\\
 &\stackrel{(a)}{\le}  q\sum_{y,z\in \Z_q} \E[\delta_u(y+z)\delta_u(z)|\xi_u=\star]^2\\
 &\stackrel{(b)}{\le} q\sum_{y,z\in \Z_q} \E[\delta_u(y+z)^2|\xi_u=\star]\E[\delta_u(z)^2|\xi_u=\star]\\
 &= q\Big(\sum_{y\in \Z_q} \E[\delta_u(y)^2|\xi_u=\star]\Big)^2.
 \end{align*}
 Inequality $(a)$ follows from $(\sum_{i=1}^n x_i)^2 \le n\sum_i x_i^2$, and $(b)$ follows from Cauchy-Schwarz. Lastly, we have $\sum_{y\in \Z_q} \E[\delta_u(y)^2|\xi_u=\star] = \E[\delta_u(\theta_u)|\xi_u=\star] = X.$ 
\end{proofof}

\vspace{.5cm}
Now we plug the estimates of Lemma~\ref{taylor_approx} in~\eqref{decomposition}. Using the fact $0 \le Z_o(x)/\sum_y Z_o(y)\le 1$, we obtain 
\[\Big| \hat{z}_{o,l} - \eps\hat{\kappa} \frac{q-1}{q} - (1-\eps)\hat{\kappa}  z_{l-1}\Big| \le C(q) \hat{\kappa}^2 (z_{l-1}^2+\eps^2).\]  
 where $C(q)$ is a constant that depends only on $q$.

\subsection{Proof of Corollary \ref{coro:Loc-Zq}}

We first prove the result concerning the overlap with $\btheta_0$. Let $\sigma \in \mathscr{S}_{q}$ be a fixed permutation.  We have 
\begin{align*}
\P\big(\hat{\theta}^{(l)}_u = \sigma(\theta_{u})\big) - \frac{1}{q} &= \sum_{x \in \Z_q}\E\big[\widehat{\mu}_{G_n,u,l}(\sigma(x)) \mu_{G_n,u}(x)\big] - \frac{1}{q}\\
&=\sum_{x \in \Z_q} \E\big[\big(\widehat{\mu}_{G_n,u,l}(\sigma(x))  - \frac{1}{q}\big)\mu_{G_n,u}(x)\big]\\
&\le \sum_{x \in \Z_q} \E\big[\big|\widehat{\mu}_{G_n,u,l}(\sigma(x))  - \frac{1}{q}\big|\big]\\
&\le \sqrt{q}\E \big[d_{\ell_2}(\widehat{\mu}_{G_n,u,l},\onu)^2\big]^{1/2}.
\end{align*}
The last line follows by Cauchy-Schwarz and then Jensen's inequality. Averaging over $u \in V_n$, applying Jensen's inequality once more, and then using Theorem~\ref{local_alg_Z_q} yields the first statement. 

Next, let $f : \Z_q \mapsto \R$ with $\sum_{x \in \Z_q} f(x)=0$ and $\frac{1}{q}\sum_{x \in \Z_q} f(x)^2=1$. The loss of $\hatbX^{(l)}$ is
\begin{align*}
\risk_n(\hatbX^{(l)};f) &= \frac{1}{n^2}\E \|\bX_f\|_F^2 -\frac{2}{n^2}\E \big\lbr \hatbX^{(l)},\bX_f\big\rbr + \frac{1}{n^2}\E \|\hatbX^{(l)}\|_F^2\\
&\ge  \frac{1}{n^2}\E \|\bX_f\|_F^2 -\frac{2}{n^2}\E \big\lbr \hatbX^{(l)},\bX_f\big\rbr.
\end{align*}
We have $\E \|\bX_f\|_F^2 = \sum_{u,v\in V_n} \E[f(\theta_u)^2f(\theta_v)^2] = \frac{n}{q}\sum_{x \in \Z_q}f(x)^4+ n(n-1)$. So $\lim \frac{1}{n^2}\E \|\bX_f\|_F^2= 1$.
On the other hand, since $\sum_{x \in \Z_q} f(x)=0$, we have 
\begin{align*}
\E\Big[f(\theta_u) \big|Y^{(\eps)}_{B_{G_n}(u,l)}\Big] &= \sum_{x \in \Z_q} \big(\widehat{\mu}_{G_n,u,l}(x)-\frac{1}{q}\big)f(x)\\
&= \sum_{x \in \Z_q} \widehat{\delta}_{u,l,G_n}(x)f(x),
\end{align*}
where $\widehat{\delta}_{u,l,G_n}(x) = \widehat{\mu}_{G_n,u,l}(x)-\frac{1}{q}$, $x \in \Z_q$. 
On the other hand we have 
\begin{align*}
\E \big\lbr \hatbX^{(l)},\bX_f\big\rbr &= \E \Big[\Big(\sum_{u \in V_n} f(\theta_u) \E\big[f(\theta_u) \big|Y^{(\eps)}_{B_{G_n}(u,l)}\big]\Big)^2\Big] \\
&= \E \Big[\Big(\sum_{x \in \Z_q}  f(x) \Big(\sum_{u \in V_n} \widehat{\delta}_{u,l,G_n}(x) f(\theta_u)\Big)\Big)^2\Big].
\end{align*}
We use Cauchy-Schwarz inequality and the fact $\sum_{x \in \Z_q} f(x)^2=q$ to obtain
\begin{align*}
\E \big\lbr \hatbX^{(l)},\bX_f\big\rbr 
&\le q \E \Big[\sum_{x \in \Z_q}  \Big(\sum_{u \in V_n} \widehat{\delta}_{u,l,G_n}(x) f(\theta_u)\Big)^2\Big]\\
&\le q \E \Big[ \Big(\sum_{u \in V_n}f(\theta_u)^2\Big)  \Big(\sum_{u \in V_n} d_{\ell_2}(\widehat{\mu}_{G_n,u,l},\onu)^2\Big)\Big]\\
&\le n q \|f\|_{\infty}^2 \sum_{u \in V_n} \E d_{\ell_2}(\widehat{\mu}_{G_n,u,l},\onu)^2.
\end{align*}
We apply Theorem~\ref{local_on_tree} to obtain 
$\limsup_{l}\limsup_{n} \frac{1}{n^2}\E \big\lbr \hatbX^{(l)},\bX_f\big\rbr \le C\|f\|_{\infty}^2\eps$,
and this yields the desired result.

\end{document}